\documentclass[12pt]{amsart}
\usepackage{epsfig,amscd,amssymb,amsxtra,amsmath,amsthm,multicol,makecell}
\newcommand{\bbQ}{{\mathbb{Q}}}
\newcommand{\bbC}{{\mathbb{C}}}
\newcommand{\bbP}{{\mathbb{P}}}

\newcommand{\bbR}{{\mathbb{R}}}
\newcommand{\bbZ}{{\mathbb{Z}}}

\newcommand{\supp}{{\mathrm{supp}}}

\newcommand{\calU}{{\mathcal{U}}}

\newcommand{\calC}{{\mathcal{C}}}

\numberwithin{equation}{section}

\newtheorem{Prop}[equation]{Proposition}
\newtheorem{Lem}[equation]{Lemma}
\newtheorem{Def}[equation]{Definition}
\newtheorem{Thm}[equation] {Theorem}
\newtheorem{Cor}[equation]{Corollary}

\newtheorem{Conj}[equation]{Conjecture}
\title
[Models of $X_0(N)$]
{On Primitive Elements of Algebraic Function Fields and Models of $X_0(N)$}

\author{Iva Kodrnja}
\address{
	Faculty of Civil Engineering, 
	University of Zagreb,
	Ka\v ci\' ceva 26, 10000 Zagreb,
	Croatia}
\email{ikodrnja@grad.hr}

\author{ Goran Mui\'c}
\address{
Department of Mathematics, 
University of Zagreb,
Bijeni\v cka 30, 10000 Zagreb,
Croatia}
 \email{gmuic@math.hr}

\begin{document}

\begin{abstract}
  This paper is a continuation of our previous works where we study maps from  $X_0(N)$, $N\ge 1$, into $\mathbb P^2$ constructed via modular forms of the same weight and criteria that
  such a map is birational (see \cite{Muic2}). In the present paper our approach is based on the theory of primitive elements in finite separable field extensions. We prove that in most of the cases
  the constructed maps are birational. We consider particular cases and their equations  in $\mathbb P^2$. 
 \end{abstract}


\maketitle

\section{Introduction} \label{intro}

Let $\Bbb H$ be the complex upper half-plane with the $SL_2(\Bbb R)$-invariant hyperbolic measure defined by $dxdy/y^2$, where the coordinates on $\Bbb H$ are written as
$z=x+\sqrt{-1}y,\: y>0$. Let $\Gamma$ be a Fuchsian group of the first kind \cite[Section 1.7, page 28]{Miyake}. By a theorem of Siegel \cite[Theorem 1.9.1]{Miyake}, a
discrete subgroup $\Gamma$ of
$SL_2(\mathbb{R})$
is  a Fuchsian group of the first kind if and only if  the hyperbolic volume of the quotient $\Gamma \setminus \Bbb H $ is finite:
$\iint_{\Gamma \setminus \Bbb H} \frac{dxdy}{y^2}<\infty$.  Examples of such groups are the important modular groups such as  $SL_2(\Bbb Z)$ and its congruence subgroups
$\Gamma_0(N)$,  $\Gamma_1(N)$, and $\Gamma(N)$  \cite[Section 4.2]{Miyake}.

The quotient $\Gamma\backslash \Bbb H$ can be compactified by adding a finite number of $\Gamma$-orbits of points in $\Bbb R\cup \{\infty\}$ called cusps of $\Gamma$ and we obtain a compact
Riemann surface which will be denoted by $\mathfrak{R}_\Gamma$. As it is an irreducible complete smooth algebraic curve, we are interested in finding its plane realizations.
Various aspects of modular curves has
been studied in \cite{BKS}, \cite{bnmjk}, \cite{sgal}, \cite{ishida}, \cite{Muic}, \cite{MuMi}, \cite{mshi} and \cite{yy}. 
We continue the approach presented in \cite{Muic1}, \cite{Muic2}, and \cite{Kodrnja1},  and start by introducing the notation and briefly repeat some results from \cite{Muic2}.
In the present paper our approach is based on the theory of primitive elements in finite separable field extensions (see \cite{ynt}, and \cite[Section 6.10]{w}).
We especially rely on the theory and  criteria for birationality
developed in \cite{Muic2}.

Assume that $\Gamma$ has at least one cusp. Let $g(\Gamma)$ be the genus of $\mathfrak{R}_\Gamma$.
For $m\geq 2$  an even integer, let $M_m(\Gamma)$ (resp., $S_m(\Gamma)$) be the space of (resp. cuspidal) modular forms of weight $m$ for $\Gamma$.  Assume $\dim M_m(\Gamma)\geq 3$.
Let $f,g,h$ be three linearly independent modular forms in $M_m(\Gamma)$. Then, we define a  holomorphic map $\mathfrak{R}_\Gamma\to \Bbb P^2$ by 
\begin{equation}\label{map}
\mathfrak{a}_z\longmapsto (f(z):g(z):h(z)).
\end{equation}
Since  $\mathfrak{R}_\Gamma$ has a canonical structure of complex projective irreducible algebraic curve, this map can be regarded as a regular map between projective varieties. Consequently, the image
is an irreducible projective curve  which we denote by $\calC(f,g,h)$.

The degree $d(f,g,h)$ of the map (\ref{map}) is by definition the degree of the field extension of the fields of rational functions:
$$
\mathbb C\left(\mathcal C(f, g, h)\right) \subset 
\mathbb C\left(\mathfrak R_\Gamma\right).
$$
This number is studied in \cite{Muic2} in great detail. The main result of \cite{Muic2} and its proof give a fairly detailed description of $d(f,g,h)$ \cite[Theorem 1-4]{Muic2}.
In order to recall that result we introduce more notation.

We recall the notion of the divisor of  $f\in M_m(\Gamma)$, $f\neq 0$ (see \cite[Section 2.3]{Miyake},  or Section \ref{prelim} in this paper).
For each $\mathfrak a\in \mathfrak R_\Gamma$, we may define
the multiplicity $\nu_{\mathfrak a}(f)$ of $f$ at  $\mathfrak a$. The multiplicity $\nu_{\mathfrak a}(f)$ is a non-negative
rational number, and, for all but finitely many points $\mathfrak a\in \mathfrak R_\Gamma$,  
$\nu_{\mathfrak a}(f)=0$. We may define the divisor of $f$ as follows:
$
\mathrm{div}{(f)}=\sum_{\mathfrak a\in \mathfrak
	R_\Gamma} \nu_{\mathfrak a}(f) \mathfrak a$. The degree of this divisor is given by
$$
\mathrm{deg}(\mathrm{div}{(f)})\overset{def}{=}\sum_{\mathfrak a\in \mathfrak
	R_\Gamma} \nu_{\mathfrak a}(f)= \frac{m}{4\pi} \iint _{\Gamma \backslash \mathbb H} \frac{dxdy}{y^2}.
$$
If $-1\in \Gamma$ and  $\Gamma$ is a subgroup of finite index in $SL_2(\mathbb Z)$, then the right-hand side is given by
the following well-known expression: $\frac{m}{12}[SL_2(\mathbb Z): \Gamma]$.

Now, \cite[Theorem 1-4]{Muic2} gives the following equality: 

$$
d(f, g, h) \cdot \deg{\mathcal C(f, g, h)}= \frac{m}{4\pi} \iint _{\Gamma \backslash \mathbb H} \frac{dxdy}{y^2}
- \sum_{\mathfrak a\in \mathfrak R_\Gamma} 
\min{\left(\nu_{\mathfrak a}(f), \nu_{\mathfrak a}(g), \nu_{\mathfrak a}(h)\right)}.
$$
Here, $\deg{\mathcal C(f, g, h)}$ is the degree of the reduced homogeneous equation defining  $\mathcal C(f, g, h)$ in $\mathbb P^2$.

In \cite[Corollary 1.5]{Muic2}, this was  further refined as follows (recall that $m\ge 2$ is even):

\begin{equation}\label{iii111}
\begin{aligned}
&d(f, g, h) \cdot \deg{\mathcal C(f, g, h)} =\\
&\begin{cases}
\dim M_m(\Gamma) + g(\Gamma) -1 - \sum_{\mathfrak a\in \mathfrak R_\Gamma} 
\min{\left(\mathfrak c'_{f}(\mathfrak a),  \mathfrak c'_{g}(\mathfrak a), 
	\mathfrak c'_{h}(\mathfrak a) \right)}, \\
\dim S_m(\Gamma) + g(\Gamma) -1 -\epsilon_m - \sum_{\mathfrak a\in \mathfrak R_\Gamma} 
\min{\left(\mathfrak c_{f}(\mathfrak a),  \mathfrak c_{g}(\mathfrak a), 
	\mathfrak c_{h}(\mathfrak a) \right)}, \\ \qquad \text{if $f, g, h\in S_m(\Gamma)$},
\end{cases}
\end{aligned}
\end{equation}
where $\epsilon_2=1$ and $\epsilon_m=0$ for $m$ even, $m\ge 4$.  Here, for example, $\mathfrak c'_{f}$ denotes the 
integral effective divisor on $\mathfrak R_\Gamma$  obtained from $\mathrm{div}{(f)}$ 
by subtracting   necessary contributions at 
elliptic points, and, in addition, if $f\in S_m(\Gamma)$, then we  subtract necessary contribution from $\mathfrak c'_{f}$ at cusps, to get a 
divisor  $\mathfrak c_{f}$. Details are standard, and they  can be found in (see \cite[Lemma 2.2]{Muic2}, or Lemma \ref{prelim-1} in this paper).

We need the following definition before we state the main result of the paper:

\begin{Def}\label{gfr-def}
Let $W\subset M_m(\Gamma)$ be a non-zero linear subspace. 
Then,  we say that  $W$ {\bf determines  the field of rational functions}  $\mathbb C(\mathfrak R_\Gamma)$ if $\dim W\ge 2$, and 
there exists  a basis  $f_0, \ldots, f_{s-1}$ of $W$,  such that $\mathbb C(\mathfrak R_\Gamma)$ is generated over $\mathbb C$ by the
quotients $f_i/f_0$, $1\le i\le s-1$.
\end{Def}
Clearly, this notion does not  depend on the choice of the basis used. Also, it is equivalent to the fact that
the holomorphic map  $\mathfrak R_\Gamma\longrightarrow \bbP^{s-1}$
given by $\mathfrak a_z\mapsto \left(f_0(z): \cdots : f_{s-1}(z)\right)$ is birational onto its image in $\bbP^{s-1}$.

For example, if $\dim S_m(\Gamma) \ge \max{(g(\Gamma)+2, 3)}$, then we can take $W=S_m(\Gamma)$ by general theory of algebraic curves \cite[Corollary 3.4]{Muic1}.
 We recall that $\mathfrak R_\Gamma$ is hyperelliptic if $g(\Gamma)\ge 2$, and there is a degree two map onto $\mathbb P^1$.
By general theory \cite[Chapter VII, Proposition 1.10]{Miranda}, if $g(\Gamma)=2$, then $\mathfrak R_\Gamma$ is hyperelliptic.
If $\mathfrak R_\Gamma$ is not hyperelliptic, then $\dim S_2(\Gamma)= g(\Gamma)\ge 3$, and we can take $W=S_2(\Gamma)$ using the fact that the holomorphic map
$\mathfrak R_\Gamma\longrightarrow \mathbb P^{g(\Gamma)-1}$ attached to a canonical divisor is an isomorphism (and in particular birational equivalence) onto its image
\cite[Chapter VII, Proposition 2.1]{Miranda}, and the possibility to interpret cuspidal forms in $S_2(\Gamma)$ as holomorphic $1$-forms on $\mathfrak R_\Gamma$
\cite[Theorem 2.3.2]{Miyake}.

Now, the first main result of the present paper is the following theorem:

\begin{Thm}\label{thm-3}
	Assume that  $m\ge 2$ is an even integer. Let $W\subset M_m(\Gamma)$, $\dim W\ge 3$,  be a subspace which determines 
	the field of rational 	functions $\mathbb C(\mathfrak R_\Gamma)$ (see Definition \ref{gfr-def}). Let $f, g \in W$ be linearly independent.
	Then there exists  a non-empty Zariski open set $\calU\subset W$ such that for any 
	$h\in \calU$ we have the following:
	\begin{itemize}
		\item[a)] $f, g,$ and $h$ are linearly independent;
		\item[b)] $\mathfrak R_\Gamma$ is birationally equivalent to  $\calC (f, g, h)$ via the map (\ref{map}).
	\end{itemize}
	Moreover, for each $h\in \mathcal U$, the degree of $\calC (f, g, h)$ is given by (\ref{iii111}) with $d(f, g, h)=1$.
        
\end{Thm}

\vskip .2in
We prove Theorem \ref{thm-3} in Section \ref{gen}. For ''generic'' modular form $h$ stated in Theorem \ref{thm-3},
the degree is always given by
$$
\deg{\mathcal C(f, g, h)} =\begin{cases}
\dim M_m(\Gamma) + g(\Gamma) -1, \\
\dim S_m(\Gamma) + g(\Gamma) -1 -\epsilon_m, \ \ \text{if $f, g, h\in S_m(\Gamma)$}.
\end{cases}
$$
This was proved in Corollary \ref{gen-1000}. But as the referee suggested, the degree can be significantly lowered as we demonstrated by various
results and examples in Section \ref{r-exam}.

In Section \ref{exist} we prove the following corollary  of Theorem \ref{thm-3}.
We recall that  $g(\Gamma_0(N))\ge 2$ unless
$$
\begin{cases}
N\in\{1-10, 12, 13, 16, 18, 25\} \ \ \text{when $g(\Gamma_0(N))=0$, and}\\
N\in\{11, 14, 15, 17, 19-21, 24, 27, 32, 36, 49\} \ \ \text{when
	$g(\Gamma_0(N))=1$.}
\end{cases}
$$
Let $g(\Gamma_0(N))\ge 2$. Then,  we remark that Ogg \cite{Ogg} has determined all $X_0(N)$ which are hyperelliptic curves.
In view of Ogg's paper, we see that $X_0(N)$ is not hyperelliptic for
$N\in \{34, 38, 42, 43, 44, 45, 51-58, 60-70\}$ or $N\ge 72$. This implies $g(\Gamma_0(N))\ge 3$.

\begin{Cor}\label{exist-1}
  Let $m\ge 2$ be an even integer.    Assume that one of the following holds:
\begin{enumerate}
\item[(A)] $g(\Gamma_0(N))\ge 1$,  and $m\ge 4$ (if $N\neq 11$)  or $m\ge 6$ (if $N=11$);
\item[(B)] $X_0(N)$ is not hyperelliptic, and $m= 2$. 
\end{enumerate}
(In either case, $\dim S_m(\Gamma_0(N))\ge 3$.)  Let $f, g \in S_m(\Gamma_0(N))$  be linearly independent with integral $q$-expansions. Then, there exists infinitely many
$h\in S_m(\Gamma_0(N))$  with integral $q$--expansion such that  we have the following:
	\begin{itemize}
		\item[(i)] $X_0(N)$ is birationally equivalent to  $\calC (f, g, h)$ via the map (\ref{map}), and
		\item[(ii)] the reduced  equation of $\calC(f, g, h)$ has integral coefficients up to a multiplication by  a non-zero constant in $\mathbb C$.
	\end{itemize}
\end{Cor}
Examples and improvements to Corollary \ref{exist-1} are included in Section \ref{r-exam} as we already mentioned above.

The rest of the paper is based on the other practical use of Theorem \ref{thm-3}. The proof of Theorem  \ref{thm-3} essentially is about
the theoretical construction of primitive elements in the finite field extension  $\bbC (g/f)\subset \bbC (\mathfrak{R}_\Gamma)$.
Methods used in the proofs of Theorem \ref{thm-3} and its Corollary \ref{gen-1000} are great theoretical tools, but not fertile result-wise. Therefore, we looked for methods of
determining primitive elements in a more direct way.

In Section \ref{epe} we discuss the special case of Theorem \ref{thm-3} when $\dim W=4$. The approach is based on estimates based on
the Primitive Element Theorem of finite separable field extensions in the form stated in \cite[Section 6.10]{w}  adapted to our case via general Lemma \ref{epe-1}, and
estimates on absolute values of roots of polynomials (see Lemma \ref{epe-4}), one of them is Mahler's estimate \cite{Mahler}.  The main results are Propositions \ref{epe-5} and  
\ref{epe-6}. Proposition \ref{epe-5} is a general result, and Proposition \ref{epe-6} is a nice example for $W=S_4(\Gamma_0(14))$.   The application of Proposition \ref{epe-6} is given by
    Corollary \ref{epe-6-cor}

In Section \ref{trial}  we adapt  to our case the trial method, commonly used in the cases of algebraic number fields, \cite{ynt},
where an element that is chosen from a certain subset of the field extension is tested for being primitive. We present a very efficient algorithm for computing model of
$X_0(N)$ when  $X_0(N)$ is not hyperelliptic and $g(\Gamma_0(N))\ge 4$. We use $W=S_2(\Gamma_0(N))$. As an example, we consider the case $X_0(72)$.

We would like to thank the referee for suggestions on improvements of our results and methods. 

\section{Preliminaries}\label{prelim}

In this section we recall necessary facts about modular forms and their divisors \cite{Miyake}.
Let $\mathbb H$ be the upper half--plane.
Then the group $SL_2(\bbR)$  acts on $\mathbb H$  as follows:
$$ g.z=\frac{az+b}{cz+d}, \ \ g=\left(\begin{matrix}a & b\\ c & d
\end{matrix}\right)\in SL_2(\bbR).
$$
We let $j(g, z)=cz+d$. The function $j$ satisfies the cocycle identity:
\begin{equation}\label{cocycle}
j(gg', z)=j(g, g'.z) j(g', z).
\end{equation}

Next, the $SL_2(\bbR)$--invariant measure on $\mathbb H$ is defined by $dx dy
/y^2$, where the coordinates on $\mathbb H$ are written in a usual way 
$z=x+\sqrt{-1}y$, $y>0$. A discrete subgroup $\Gamma\subset
SL_2(\bbR)$ is called a Fuchsian group of the first kind if 
$$
\iint _{\Gamma \backslash \mathbb H} \frac{dxdy}{y^2}< \infty.
$$
Then, adding a finite number of points 
in $\bbR\cup \{\infty\}$, called cusps, $\mathcal F_\Gamma$ can be
compactified. In this way we obtain a compact Riemann surface 
$\mathfrak R_\Gamma$. One of the most important examples are the groups
$$
\Gamma_0(N)=\left\{\left(\begin{matrix}a & b \\ c & d\end{matrix}\right) \in SL_2(\mathbb Z); \ \
c\equiv 0 \ (\mathrm{mod} \ N)\right\}, \ \ N\ge 1.
$$
We write $X_0(N)$ for    $\mathfrak R_{\Gamma_0(N)}$.

Let $\Gamma$ be a Fuchsian group of the first kind. 
Let $m\ge 2$ be an even integer.  We consider the space $M_m(\Gamma)$
(resp., $S_m(\Gamma)$)  of all 
modular (resp., cuspidal) forms of weight $m$.
We also need the following obvious property: for 
$f, g\in M_m(\Gamma)$, $g\neq 0$, the quotient $f/g$ is a meromorphic function on  $\mathfrak R_{\Gamma}$.

Next, we recall from \cite[Section 2.3]{Miyake} some notions related to the theory of divisors 
of modular forms of even weight $m\ge 2$ and state a preliminary result.

Let $m\ge 2$ be an even integer and $f\in M_{m}(\Gamma)-\{0\}$. Then,
$\nu_{z-\xi}(f)$ denotes the order of the holomorphic function $f$ at $\xi$.
For each $\gamma\in \Gamma$, the functional equation
$f(\gamma.z)=j(\gamma, z)^m f(z)$, $z\in \mathbb H$, shows that 
$\nu_{z-\xi}(f)=\nu_{z-\xi'}(f)$ where $\xi'=\gamma.\xi$. 
Also, if we let
$$
e_{\xi} =\# \left(\Gamma_{\xi}/\Gamma\cap \{\pm 1\}\right),
$$
then $e_{\xi}=e_{\xi'}$, where $\Gamma_{\xi}$ is the stabilizer of $\xi$ in $\Gamma$.
The point $\xi\in \mathbb H$ is elliptic if $e_\xi>1$. Next, following \cite[Section 2.3]{Miyake}, we define
$$
\nu_\xi(f)=\nu_{z-\xi}(f)/e_{\xi}.
$$
Clearly, $\nu_{\xi}=\nu_{\xi'}$, and we may let 
$$
\nu_{\mathfrak a_\xi}(f)=\nu_\xi(f),
$$ 
where 
$$\text{$\mathfrak a_\xi\in \mathfrak R_\Gamma$ is the projection of $\xi$
	to $\mathfrak R_\Gamma$,} 
$$
a notation we use throughout this paper.

If $x\in \bbR\cup \{\infty\}$ is a cusp for $\Gamma$, then we define 
$\nu_x(f)$ as follows. Let $\sigma\in SL_2(\bbR)$ such that
$\sigma.x=\infty$. We write 
$$
\{\pm 1\} \sigma \Gamma_{x}\sigma^{-1}= \{\pm 1\}\left\{\left(\begin{matrix}1 & lh'\\ 0 &
1\end{matrix}\right); \ \ l\in \bbZ\right\},
$$
where $h'>0$. Then we write the Fourier expansion of $f$ at $x$ as follows:
$$
(f|_m \sigma^{-1})(\sigma.z)= \sum_{n=0}^\infty a_n e^{2\pi
	\sqrt{-1}n\sigma.z/h'}.
$$

We let 
$$
\nu_x(f)=l\ge 0,
$$
where $l$ is defined by $a_0=a_1=\cdots =a_{l-1}=0$, $a_l\neq 0$. One
easily see that this definition does not depend on $\sigma$. Also, 
if $x'=\gamma.x$, then
$\nu_{x'}(f)=\nu_{x}(f)$. Hence, if $\mathfrak b_x\in \mathfrak
R_\Gamma$ is a cusp corresponding to $x$, then we may define
$$
\nu_{\mathfrak b_x}(f)=\nu_x(f). 
$$

Put
$$
\mathrm{div}{(f)}=\sum_{\mathfrak a\in \mathfrak
	R_\Gamma} \nu_{\mathfrak a}(f) \mathfrak a \in \ \  \bbQ\otimes \mathrm{Div}(\mathfrak R_\Gamma),
$$
where $\mathrm{Div}(\mathfrak R_\Gamma)$ is the group of (integral) divisors on 
$\mathfrak R_\Gamma$.

Using \cite[Section 2.3]{Miyake}, this sum is finite i.e., $ \nu_{\mathfrak a}(f)\neq 0$
for only a finitely many points. We let 
$$
\mathrm{deg}(\mathrm{div}{(f)})=\sum_{\mathfrak a\in \mathfrak
	R_\Gamma} \nu_{\mathfrak a}(f).
$$

Let $\mathfrak d_i\in \bbQ\otimes \mathrm{Div}(\mathfrak R_\Gamma)$, $i=1, 2$. Then we say 
that $\mathfrak d_1\ge \mathfrak d_2$ if their difference $\mathfrak d_1 - \mathfrak d_2$ 
belongs to $\mathrm{Div}(\mathfrak R_\Gamma)$ and is non-negative in the usual sense.

\begin{Lem}\label{prelim-1} Assume that $m\ge 2$ is an even integer. Assume that 
	$f\in  M_m(\Gamma)$, $f\neq 0$. Let $t$ be the number of inequivalent cusps  for $\Gamma$.  Then we have the following:
	\begin{itemize}
		
		\item[(i)] For $\mathfrak a\in \mathfrak
		R_\Gamma$, we have  $\nu_{\mathfrak a}(f) \ge 0$.
		
		\item [(ii)]  For a cusp $\mathfrak a\in \mathfrak R_\Gamma$, we have that 
		$\nu_{\mathfrak a}(f)\ge 0$ is an integer.

		\item[(iii)] If  $\mathfrak a\in \mathfrak
		R_\Gamma$ is not an elliptic point or a cusp, then $\nu_{\mathfrak a}(f)\ge 0$
		is an integer.  If  $\mathfrak a\in \mathfrak
		R_\Gamma$ is an elliptic point, then $\nu_{\mathfrak a}(f)-\frac{m}{2}(1-1/e_{\mathfrak a})$ is 
		an integer.

		\item[(iv)]Let $g(\Gamma)$ be the genus of $ \mathfrak R_\Gamma$. Then 
		\begin{align*}
		\mathrm{deg}(\mathrm{div}{(f)})&= m(g(\Gamma)-1)+ \frac{m}{2}\left(t+ \sum_{\mathfrak a\in \mathfrak
			R_\Gamma, \ \ elliptic} (1-1/e_{\mathfrak a})\right)\\
		&= \frac{m}{4\pi} \iint_{\Gamma \backslash \mathbb H} \frac{dxdy}{y^2}.
		\end{align*}

		\item[(v)] Let $[x]$ denote the largest integer $\le x$ for $x\in
		\bbR$.  Then
		
		\begin{align*}
		\dim S_m(\Gamma) &=
		\begin{cases} (m-1)(g(\Gamma)-1)+(\frac{m}{2}-1)t+ \sum_{\substack{\mathfrak a\in \mathfrak
				R_\Gamma, \\ elliptic}} \left[\frac{m}{2}(1-1/e_{\mathfrak a})\right], \\ \qquad  \text{if $m\ge 4$,}\\
		g(\Gamma), \ \ \text{if $m=2$.}\\
		\end{cases}\\
		\dim M_m(\Gamma)&=\begin{cases} \dim S_m(\Gamma)+t, \ \ \text{if $m\ge 4$, or $m=2$ and $t=0$,}\\
		\dim S_m(\Gamma)+t-1=g(\Gamma)+t-1,\ \ \text{if $m=2$ and $t\ge 1$.}\\
		\end{cases} 
		\end{align*}
		
	      \item[(vi)] Let $\mathfrak c'_f$ be defined by
                 \begin{align*}
		\mathfrak c'_f=& \mathrm{div}{(f)}- \sum_{\mathfrak a\in \mathfrak
			R_\Gamma, \ \ elliptic} \left(\frac{m}{2}(1-1/e_{\mathfrak a}) -
		\left[\frac{m}{2}(1-1/e_{\mathfrak
			a})\right]\right)\mathfrak a.
		\end{align*}
                 Then $\mathfrak c'_f$ is an integral effective divisor of degree
		$$
		\begin{cases}
		\dim M_m(\Gamma)+ g(\Gamma)-1, \ \ \text{if $m\ge 4$, or $m=2$ and $t\ge 1$,}\\
		2(g(\Gamma)-1), \ \ \text{if $m=2$ and $t=0$}
		\end{cases}.
		$$
	        \item[(vii)] Assume that $f\in S_m(\Gamma)$. Then, the  integral divisor defined by
		$ \mathfrak c_f\overset{def}{=}\mathfrak c'_f-
		\sum_{\substack{\mathfrak b \in \mathfrak
				R_\Gamma, \\ cusp}} \mathfrak b$ satisfies  $\mathfrak c_f\ge 0$  and its degree is given by
		$$
		\begin{cases}
		\dim S_m(\Gamma)+ g(\Gamma)-1; \ \ \text{if $m\ge 4$,}\\
		2(g(\Gamma)-1); \ \ \text{if $m=2$.}
		\end{cases}
		$$
	\end{itemize}
\end{Lem}
\begin{proof} The claims (i)--(v) are standard \cite[Sections 2.3, 2.4, 2.5]{Miyake}. The claim (vi) follows from (iii), 
	(iv), and (v) (see \cite[Lemma 4-1]{Muic}). Finally, (vii) follows from (vi).
\end{proof}

\section{The proof of Theorem \ref{thm-3} and a Corollary}
\label{gen}

We begin the proof of Theorem \ref{thm-3} with the following lemma:

\begin{Lem}\label{p2-8}  
	Assume that  $m\ge 2$ is an integer. Let $W\subset M_m(\Gamma)$, $\dim W\ge 3$,  be a subspace which determines
	the field of rational 
	functions $\mathbb C(\mathfrak R_\Gamma)$ (see Definition \ref{gfr-def}).  Let $f, g\in W$ be linearly independent.
	Then, there exists  a non-empty Zariski open set $\calU\subset W$ such that for $h\in \calU$ we have the following:
	\begin{itemize}
		\item[a)] $f, g,$ and $h$ are linearly independent;
		\item[b)] the field of rational functions
		$\Bbb C (\mathfrak R_\Gamma)$ is generated over $\Bbb  C$ by 
		$g/f$ and $h/f$.
	\end{itemize}
\end{Lem}
\begin{proof} 	We select a basis $f_0, \ldots, f_{s-1}$ of $W$ such that 
	$f=f_0$ and $g=f_1$. By the assumption on $W$, the field of rational functions
	$\Bbb C (\mathfrak R_\Gamma)$ is generated over $\Bbb  C$ by all $f_i/f_0$, $1\le i\le s$. 
	We let 
	$$
	K=\Bbb C(f_1/f_0),
	$$
	and 
	$$
	L= \Bbb C (\mathfrak R_\Gamma)=
	\Bbb C(f_1/f_0, \ldots, f_{s-1}/f_0)=K(f_2/f_0, \ldots, f_{s-1}/f_0).
	$$
	Since $L$ has transcendence degree $1$ over $\mathbb C$, 
	$f_2/f_0, \ldots, f_{s-1}/f_0$ are all algebraic over $K$. Thus, the field $L$ 
	is a finite algebraic extension of $K$. It is also obviously separable.

        \begin{Lem}\label{p2-8001}  There exists  $(\lambda_2, \ldots, \lambda_{s-1})\in \Bbb C^{s-2}$  such that  $(\lambda_2 f_2+\cdots + \lambda_{s-1} f_{s-1})/f_0$ generates 
	  $L$ over $K$ i.e.,
          $$
	 L= K((\lambda_2 f_2+\cdots + \lambda_{s-1} f_{s-1})/f_0)=
	\Bbb C(f_1/f_0, (\lambda_2 f_2+\cdots + \lambda_{s-1} f_{s-1})/f_0).
	$$
        \end{Lem}
        \begin{proof} Since $\mathbb C$ is a subfield of $K$, this follows using a variant of a proof of Primitive Element Theorem given by \cite[Section 6.10]{w}
        (see also the second paragraph in Section \ref{epe}). 
	\end{proof} 

        Now, we explain a systematic way to get them all.  Let us fix  an algebraic closure $\overline{K}$ of $K$ containing $L$.
        We consider the polynomial ring $K[X]$ in variable $X$. For $x\in L$, we define a $K$--linear endomorphism $T_x(y)=xy$, and
        attach usual invariants from elementary Linear algebra: the minimal polynomial, say $\mu(X, x) \in K[X]$, and characteristic polynomial $k(X, x)=
        \det{\left(X\cdot Id_L- T_x\right)} \in K[X]$, where $Id_L$ is identity on $L$. The degree of $k(X, x)$ is $[L:K]$.

        By elementary field theory, $\mu(X, x)$ is also a unique monic irreducible polynomial of $x$ over $K$.
        Therefore, the roots  in $\overline{K}$ of $\mu(X, x)$ are all simple.  Also, by elementary Linear algebra,  $\mu(X, x)$ and  $k(X, x)$ have the same set of roots in
        $\overline{K}$, and the  multiplicity of each root of $\mu(X, x)$ is less than or equal to the multiplicity of the same root in $k(X, x)$.
        This immediately implies 

        \begin{Lem}\label{p2-8002}  Let $x \in L$. Then,  $L=K(x)$  if and only if all roots in $\overline{K}$ of $k(X, x)$  are simple.
        \end{Lem}
        \begin{proof} By elementary field theory, $L=K(x)$ if and only if  the degree of  $\mu(X, x)$ is $[L:K]$.   By above discussion, 
          this is  equivalent to the fact that
          $\mu(X, x)=k(X, x)$ since both polynomials are monic, and  $\mu(X, x)$ divides $k(X, x)$ .
          Again, by above considerations, this is  equivalent to the fact   that all roots in $\overline{K}$ of $k(X, x)$  are simple.
        \end{proof}

For $(\lambda_2, \ldots, \lambda_{s-1})\in \Bbb C^{s-2}$, we consider  the characteristic polynomial
	$$
	P(X, \lambda_2, \ldots, \lambda_{s-1})\overset{def}{=} k\left(X,  \left(\lambda_2 f_2+\cdots + \lambda_{s-1} f_{s-1}\right)/f_0\right).
	$$
The discriminant $R$ of $P(X, \lambda_2, \ldots, \lambda_{s-1})$ with respect to $X$ i.e., the resultant  with respect to the variable $X$ 
	of the polynomial $P(X, \lambda_2, \ldots, \lambda_{s-1})$ and its 
	derivative $\frac{\partial}{\partial X}P(X, \lambda_2, \ldots, \lambda_{s-1})$
	is a polynomial in $\lambda_2, \ldots, \lambda_{s-1}$  with coefficients in $K$.
        We remark that the degree of
	$P(X, \lambda_2, \ldots, \lambda_{s-1})$ is $[L:K]\ge 2$, and of
	$\frac{\partial}{\partial X}P(X, \lambda_2, \ldots, \lambda_{s-1})$ is $[L:K]-1\ge 1$. Consequently,
	both depend on $X$ as it is required in the definition of the resultant.

        \begin{Lem}\label{p2-8003}   The discriminant $R$ is not identically equal to zero. Moreover, for  $(\lambda_2, \ldots, \lambda_{s-1})\in \Bbb C^{s-2}$, 
          $R(\lambda_2, \ldots, \lambda_{s-1})\neq 0$ if and only if $(\lambda_2 f_2+\cdots + \lambda_{s-1} f_{s-1})/f_0$ generates 
	  $L$ over $K$ 
\end{Lem}
        \begin{proof}
          The last claim follows from Lemma \ref{p2-8002} and above definition of $R$. The first claim follows from the last, and Lemma \ref{p2-8001}.
   \end{proof}

	Still, the discriminant is a polynomial in variables $\lambda_2, \ldots, \lambda_{s-1}$ with coefficients in $K$. We recall that elements of the field $K$ are rational functions on
        $\mathfrak R_\Gamma$. 
        So, to obtain a polynomial  with coefficients in $\mathbb C$, we write $\mathcal P$ for the (finite) set of all poles  of
        all non--zero coefficients  of $R$.  Then, for $\mathfrak a\in \mathfrak R_\Gamma\setminus \mathcal P$, 
        $R(\lambda_2, \ldots, \lambda_{s-1})(\mathfrak a)$ is a polynomial in variables $\lambda_2, \ldots, \lambda_{s-1}$ with coefficients in $\mathbb C$. Obviously, for
        $(\lambda_2, \ldots, \lambda_{s-1})\in \Bbb C^{s-2}$, $R(\lambda_2, \ldots, \lambda_{s-1})\neq 0$ is equivalent to the fact that there exists
        $\mathfrak a\in \mathfrak R_\Gamma\setminus \mathcal P$
        such that  $R(\lambda_2, \ldots, \lambda_{s-1})(\mathfrak a)\neq 0$. Thus, the condition
        \begin{equation}\label{p2-50}
	R(\lambda_2, \ldots, \lambda_{s-1})\neq 0
	\end{equation}
        defines a Zariski open set in  $\mathbb C^{s-2}$.

        Also, by Lemma \ref{p2-8002}, if $(\lambda_2, \ldots, \lambda_{s-1})$ belongs to that Zariski open set, then  
	$$
	h\overset{def}{=}\lambda_2 f_2+\cdots + \lambda_{s-1} f_{s-1} \in \mathbb C f_2\oplus \cdots \oplus  \mathbb C f_{s-1}
	$$
        generates $L$ over $K$. 
	It does not affect the thing if we enlarge $h$ to be 
	$$
	h=\lambda_0f_0+ \lambda_1f_1+\lambda_2 f_2+\cdots + \lambda_{s-1} f_{s-1},
	$$
	where $\lambda_0, \lambda_1$ are arbitrary complex numbers.  This means that $h$ can be selected
	from the  Zariski open subset of  $W$ given by (\ref{p2-50}) in coordinates 
        $$
        W=\mathbb C f_0\oplus \cdots \oplus  \mathbb C f_{s-1}.
        $$
         We consider 
	the discriminant $R$ as a polynomial of all variables $\lambda_0, \ldots, \lambda_{s-1}$ but which does not depend 
	on the first two variables. This completes the proof of the lemma.
\end{proof}

Now, in order to complete the proof of Theorem \ref{thm-3}, we need to prove the formula for the degree of $\calC (f, g, h)$. But, since by Lemma \ref{p2-8}
the curve $\calC (f, g, h)$ is birational to $\mathfrak R_\Gamma$, we have  $d(f, g, h) =1$ as explained in the introduction. Finally, the formula for the degree
follows from \cite{Muic2} as we explained in the introduction before the statement of Theorem \ref{thm-3}.

In the following corollary to Theorem \ref{thm-3} we need the next definition. 

\begin{Def}\label{gen-1000DEF}
  Assume that  $m\ge 2$ is an even integer. Let $W\subset M_m(\Gamma)$, $W\neq 0$,  be a linear subspace.
  We say that $W$ is a base-point-free if one of the following holds:
  \begin{itemize}
\item[(i)]  $W\not\subset S_m(\Gamma)$, and,  for each $\mathfrak a\in  \mathfrak R_\Gamma$, there exists $f\in W$, $f\ne 0$, such that
  $\mathfrak c'_f(\mathfrak a)=0$ (see Lemma \ref{prelim-1} (vi) for notation), or
\item[(ii)]  $W\subset S_m(\Gamma)$, and,  for each $\mathfrak a\in  \mathfrak R_\Gamma$, there exists $f\in W$, $f\ne 0$, such that
  $\mathfrak c_f(\mathfrak a)=0$ (see Lemma \ref{prelim-1} (vii) for notation).
\end{itemize}  
\end{Def}

For example, $\dim S_m(\Gamma) \ge \max{(g(\Gamma)+2, 3)}$, which implies $m\ge 4$, then we can take $W=S_m(\Gamma)$ by general theory of algebraic curves
(see \cite[Theorem 3.3]{Muic1}). Also, if  $g(\Gamma) \ge 3$, then $W=S_2(\Gamma)$ is a  base-point-free  using isomorphism of $S_2(\Gamma)$ with the space of holomorphic differential forms
on $\mathfrak R_\Gamma$ \cite[Theorems 2.3.2 and 2.3.3]{Miyake}, and the fact that the corresponding canonical linear system is a base-point-free
\cite[Chapter VII, Lemma 1.14]{Miranda}.

\begin{Cor}\label{gen-1000} Assume that  $m\ge 2$ is an even integer. Let $W\subset M_m(\Gamma)$, $\dim W\ge 3$,  be a subspace which determines 
  the field of rational functions $\mathbb C(\mathfrak R_\Gamma)$ (see Definition \ref{gfr-def}), and is a base-point-free (see Definition \ref{gen-1000DEF}).
  Let $f, g \in W$ be linearly independent.
  Then, there exists a non-empty Zariski open set $\mathcal V \subset W$ such that for any 
	$h\in \mathcal V$ we have the following:
	\begin{equation}\label{eqdeg}
	\deg{\mathcal C(f, g, h)} =\begin{cases}
	\dim M_m(\Gamma) + g(\Gamma) -1, \ \ \text{if $W\not\subset S_m(\Gamma)$,} \\
	\dim S_m(\Gamma) + g(\Gamma) -1 -\epsilon_m, \ \ \text{if $W\subset S_m(\Gamma)$}.
	\end{cases}
	\end{equation}
\end{Cor}
\begin{proof}  We consider the case  $W\subset S_m(\Gamma)$. The other one is analogous.  Let us fix a basis $f_0, \ldots, f_{s-1}$ of $W$ such that 
	$f=f_0$ and $g=f_1$.   For $f, g \in W$, let us fix a Zariski open subset $\mathcal U$
  such that the conclusion of Theorem \ref{thm-3} holds.   Since $W$  is a base-point-free, for each 
	$\mathfrak a \in \supp{(\mathfrak c_{f_0})}$, there exists 
	$i_{\mathfrak a}\in \{1, \ldots, s-1\}$
	such that $\mathfrak a \not\in \supp{(\mathfrak c_{f_{i_{\mathfrak a}}})}$. Then, the rational functions 
	$f_i/f_{i_{\mathfrak a}}$ are defined at $\mathfrak a$ since we have 
	the following  (see Lemma \ref{prelim-1} (vi)):
	$$
	\text{div}\left(\frac{f_i}{f_{i_{\mathfrak a}}}\right)=\text{div}(f_i)-\text{div}(f_{i_{\mathfrak a}})=
	\mathfrak c_{f_i}-\mathfrak c_{f_{i_{\mathfrak a}}},
	$$
	where the rightmost is the difference of two effective divisors,
	so that the point 
	$\mathfrak a$ does not belong to the divisors of poles because of $\mathfrak a \not\in 
	\supp{(\mathfrak c_{f_{i_{\mathfrak a}}})}$.

	Now, we can form the 	following product of non-zero linear forms in $(\lambda_0, \ldots, \lambda_{s-1})\in \Bbb C^{s}$:
	\begin{equation}\label{gen-prodd}
	\prod_{\mathfrak a \in \supp{(\mathfrak c_{f_0})}} \left(\lambda_0 
	\frac{f_0}{f_{i_{\mathfrak a}}}(\mathfrak a) +
	\lambda_1 \frac{f_1}{f_{i_{\mathfrak a}}}(\mathfrak a) +\cdots + \lambda_{s-1} 
	\frac{f_{s-1}}{f_{i_{\mathfrak a}}}(\mathfrak a) 
	\right).
	\end{equation}

        Let $\mathcal U'\subset W$  be the Zariski open subset $\mathcal U'\subset W$ consisting of all 
	$\sum_{i=0}^{s-1}\lambda_i f_i \in W$ such that the  product in (\ref{gen-prodd}) is not equal to zero.
        For $\sum_{i=0}^{s-1}\lambda_i f_i \in \mathcal U'$,           
	neither of $\mathfrak a \in \supp{(\mathfrak c_{f_0})}$ belong to the divisor of zeros 
	$\text{div}_0\left((\sum_{i=0}^{s-1} \lambda_i f_i)/f_{i_{\mathfrak a}}\right)$ of the corresponding rational function.
       Since $\mathfrak a \not\in \supp{(\mathfrak c_{f_{i_{\mathfrak a}}})}$,  and 
	$$
	\text{div}_0\left(\frac{\sum_{i=0}^{s-1}\lambda_i f_i}{f_{i_{\mathfrak a}}}\right)
	-\text{div}_\infty\left(\frac{\sum_{i=0}^{s-1}\lambda_i f_i}{f_{i_{\mathfrak a}}}\right)=
	\text{div}\left(\frac{\sum_{i=0}^{s-1}\lambda_i f_i}{f_{i_{\mathfrak a}}}\right)=
	\mathfrak c_{\sum_{i=0}^{s-1}\lambda_i f_i}-\mathfrak c_{f_{i_{\mathfrak a}}},
	$$
	where the rightmost  expression is the difference of two effective divisors, 
	we obtain

        \begin{equation}\label{gen-prodd-1}
	\mathfrak a \in \supp{(\mathfrak c_{f_0})} \implies 
	\mathfrak a \not\in \supp{(\mathfrak c_{\lambda_0  f_0+\lambda_1 f_1+\cdots + 
			\lambda_{s-1} f_{s-1}})}.
	\end{equation}

        Finally, we define the Zarski open subset $\mathcal U$ by $\mathcal V=\mathcal U\cap \mathcal U'$.
        For $h$ defined by $h=\sum_{i=0}^{s-1}\lambda_i f_i$ in $\mathcal V$, we have
       $$ \deg{\mathcal C(f, g, h)} =\dim S_m(\Gamma) + g(\Gamma) -1 -\epsilon_m - \sum_{\mathfrak a\in \mathfrak R_\Gamma} 
	\min{\left(\mathfrak c_{f}(\mathfrak a),  \mathfrak c_{g}(\mathfrak a), 
	  \mathfrak c_{h}(\mathfrak a) \right)}$$
        by Theorem \ref{thm-3}. Now, since $f=f_0$, using (\ref{gen-prodd-1}), we obtain
        $$
        \sum_{\mathfrak a\in \mathfrak R_\Gamma} 
	\min{\left(\mathfrak c_{f}(\mathfrak a),  \mathfrak c_{g}(\mathfrak a), 
	  \mathfrak c_{h}(\mathfrak a) \right)}=0,
        $$
        proving the corollary. 
\end{proof}

\section{Proof of Corollary \ref{exist-1}}\label{exist}
In this section we use the following notation: $\nu_\infty(\Gamma_0(N))$ is the number of inequivalent cusps, $\nu_2(\Gamma_0(N))$ (resp., $\nu_3(\Gamma_0(N))$) is the number of
inequivalent elliptic points of order $2$ (resp. $3$) of the congruence subgroup $\Gamma_0(N)$.

We begin with the proof of  Corollary \ref{exist-1} assuming that (A) holds (see the statement of  Corollary \ref{exist-1}).
First, we need to assure that $W=S_m(\Gamma_0(N))$ determines the field of rational functions  $\mathbb C(X_0(N))$ (see Theorem \ref{thm-3}).
By general theory of algebraic curves \cite[ Corollary 3.7]{Muic1}, it is enough to require 

$$
\dim S_m(\Gamma_0(N))\ge \max{(g(\Gamma_0(N))+2, 3)}.
$$
Since we assume that
$$
g(\Gamma_0(N))\ge 1,
$$
We only need to  require that

$$
\dim S_m(\Gamma_0(N))\ge g(\Gamma_0(N))+2.
$$
Using   Lemma 2-2 (v),  we obtain

\begin{multline*}
\dim S_m(\Gamma_0(N)) =
(m-1)(g(\Gamma_0(N))-1)+\left(\frac{m}{2}-1\right)\nu_\infty(\Gamma_0(N)) +\\
+ \left[\frac{m}{4}\right]\nu_2(\Gamma_0(N))
+\left[\frac{m}{3}\right]\nu_3(\Gamma_0(N)).
\end{multline*}
By \cite[Theorem 4.2.7]{Miyake}, we have 
$$
\nu_\infty(\Gamma_0(N))=\sum_{d>0, d|N} \phi((d, N/d))\ge 3
$$
unless $N$ is prime number in which case $\nu_\infty(\Gamma_0(N))=2$. Next, unless 
$\nu_2(\Gamma_0(N))=\nu_3(\Gamma_0(N))=0$, above formula shows that  for $m=4$ we have 
$$
\dim S_4(\Gamma_0(N))\ge 3(g(\Gamma_0(N))-1)+ 2\left(\frac{4}{2}-1\right)+ 1=3 g(\Gamma_0(N)
\ge g(\Gamma_0(N))+2, 
$$
since we assume that $g(\Gamma_0(N)\ge 1$. Then, for an even $m\ge 6$, we have 
\begin{multline*}
\dim S_m(\Gamma_0(N))\ge \dim \left(S_2(\Gamma_0(N)\cdot  S_{m-2}(\Gamma_0(N)\right)\ge \cdots \ge \\
\ge \dim \left(S_2(\Gamma_0(N)\cdot  S_{4}(\Gamma_0(N)\right)
\ge \dim S_4(\Gamma_0(N))\ge g(\Gamma_0(N))+2.
\end{multline*}

Similarly,  the same inequality holds if $\nu_2(\Gamma_0(N))=\nu_3(\Gamma_0(N))=0$ but 
$N$ is not a prime. It remains to consider the case $N$ is a prime and  $\nu_2(\Gamma_0(N))=\nu_3(\Gamma_0(N))=0$. 
In this case 
$$
\dim S_4(\Gamma_0(N))= 3g(\Gamma_0(N))-1\ge g(\Gamma_0(N))+2 
$$ 
if and only if $g(\Gamma_0(N))\ge 2$. It remains to consider the case 
$N$ is prime, $\nu_2(\Gamma_0(N))=\nu_3(\Gamma_0(N))=0$, and $g(\Gamma_0(N))=1$. In this case \cite[Theorem 4.2.11]{Miyake} 
gives us $[SL_2(\mathbb Z): \Gamma_0(N)]=12$. Applying \cite[Theorem 4.2.5]{Miyake}, we see that 
$\psi(N)=N+1=12$ since $N$ is prime. Hence, $N=11$. In this case, we use \cite[Theorem 4.2.5]{Miyake} to check that we 
indeed have $\nu_2(\Gamma_0(11))=\nu_3(\Gamma_0(11))=0$ and $g(\Gamma_0(11))=1$. This gives us 
$\dim S_4(\Gamma_0(11)) = 2$ and
$$
\dim S_6(\Gamma_0(11)) =4> 3=g(\Gamma_0(6))+2.
$$
Again, for an even $m\ge 8$, we have 
\begin{multline*}
\dim S_m(\Gamma_0(11))\ge \dim \left(S_2(\Gamma_0(11)\cdot  S_{m-2}(\Gamma_0(11)\right)\ge \cdots \ge \\
\ge \dim \left(S_2(\Gamma_0(11)\cdot  S_{6}(\Gamma_0(11)\right)
\ge \dim S_6(\Gamma_0(11))\ge g(\Gamma_0(11))+2.
\end{multline*}

Above considerations show that we can apply Theorem \ref{thm-3}. Next, by Eichler--Shimura theory \cite[Theorem 3.5.2]{shi}, for 
each even integer $m\ge 2$ the space of cusp forms $S_m(\Gamma_0(N))$ has a basis as a complex vector space 
consisting of forms which have integral $q$--expansions. So, if we have
$f, g \in S_m(\Gamma_0(N))$  with  integral coefficients in their $q$--expansions, then we can select infinitely many $h$ which also
have integral coefficients in  their $q$--expansions in the set $\mathcal U$ for $f$ and $g$ (see Theorem \ref{thm-3}).         
This is because $\mathbb Z^l$ is Zariski dense in $\mathbb C^l$  for any $l\ge 1$.  This proves (i).

In order to prove (ii),  we write a reduced equation of $\calC (f, g, h)$ as follows: 
$$
P=\sum_{\substack{\alpha=(\alpha_1,\alpha_2, \alpha_3)\in \mathbb Z^3_{\ge 0}\\ |\alpha|\overset{def}{=}\alpha_1+\alpha_2+\alpha_3=l}} \ a_{\alpha} x_0^{\alpha_1} x_1^{\alpha_2} x_2^{\alpha_3},
$$
where $x_0, x_1,$ and $x_2$ are variables, coefficients $a_\alpha\in \mathbb C$, and
$$
l=\deg{\mathcal C(f, g, h)}.
$$

Next, since $f, g, h\in S_m(\Gamma_0(N))$, we have that
$$
f^{\alpha_1} g^{\alpha_2} h^{\alpha_3} \in S_{lm}(\Gamma_0(N)),
$$
and those forms have integral $q$--expansions for all $\alpha\in \mathbb Z^3_{\ge 0}$ such that $|\alpha|=l$. Let us write their $q$--expansions as follows:
$$
f^{\alpha_1}(z) g^{\alpha_2}(z) h^{\alpha_3}(z)= \sum_{n=1}^\infty b_n^\alpha q^n.
$$
Then,
$$
P(f(z), g(z), h(z))=0, \ \ \text{for all $z\in \mathbb H$}
$$
is equivalent to an infinite system of homogeneous equations for coefficients $a_\alpha$ given by

$$
\sum_{\substack{\alpha=(\alpha_1,\alpha_2, \alpha_3)\in \mathbb Z^3_{\ge 0}\\ |\alpha|\overset{def}{=}\alpha_1+\alpha_2+\alpha_3=l}} \ b_n^\alpha a_{\alpha} =0, \ \ n\ge 1.
$$

By elementary theory of algebraic curves, this homogeneous system has, up to a multiplication by a non-zero constant in $\mathbb C$, unique solution in complex numbers for coefficients
$a_\alpha$. By the theory of homogeneous systems, this means that its rank is $l-1$. But coefficients $b_n^\alpha$ are all integers, and the system has rank $l-1$. This means that
the  homogeneous system has, up to multiplication by a non-zero constant in $\mathbb Q$, unique solution in integers  for coefficients
$a_\alpha$. This proves (ii). This completes the proof of Corollary \ref{exist-1}.

The proof of Corollary \ref{exist-1} assuming (B) is very similar. First, we have
$$
\dim S_2(\Gamma_0(N)) =g(\Gamma_0(N))\ge 3.
$$
Also, since $X_0(N)$ is not hyperelliptic, $W=S_2(\Gamma_0(N))$  determines the field of rational functions  $\mathbb C(X_0(N))$ by general theory of algebraic curves
\cite[Chapter VII.2, Proposition 2.1]{Miranda}. Now, we complete the proof in the same way as we completed the proof of Corollary \ref{exist-1} assuming (A).

\section{Examples and Improvements}\label{r-exam}

In this section  we give examples of equations of the corresponding curves using SAGE and compute degrees.
It is demonstrated how useful is the theory developed in \cite{Muic2} especially the test for birational equivalence
stated in the Introduction of \cite{Muic2}. 

\vskip .2in
The following example came out from the remarks of the referee. The $q$--expansions are computed using SAGE.

\begin{Prop} \label{r-exam-1}
	Consider three linearly independent forms from the four dimensional space $S_4(\Gamma_0(14))$  of cusp forms of weight four for $\Gamma_0(14)$: 
	\begin{align*}
	f& = q^2- 2q^5-2q^6 + q^7-6q^8 + 12q^{10} + 4q^{11} + 2q^{13}-5q^{14} + 4q^{15} + 10q^{16} + \cdots,\\
	g &= q^3-q^5-2q^6-q^7-4q^8 + 6q^9 + 10q^{10}-6q^{11} + 4q^{12}-3q^{13}-2q^{14} + \cdots,\\
	h &= q^4-2q^5 + q^7 + q^8-4q^{10} + 4q^{11}-2q^{12} + 2q^{13} + 2q^{14} + 4q^{15}-5q^{16} + \cdots.
	\end{align*}
	Then, the map (\ref{map}) is a birational equivalence of $X_0(14)$ and $\calC (f, g, h)$. Moreover, $\deg{\mathcal C(f, g, h)} =3$. 
\end{Prop}
\begin{proof} Let  $\mathfrak a_\infty$ be the $\Gamma_0(14)$--orbit of the cusp $\infty$. 
	Since the forms have at least double zero at $\mathfrak a_\infty$, and $f$ has exactly double zero, we have 
	$$\sum_{\mathfrak a\in X_0(14)} 
	\min{\left(\mathfrak c_{f}(\mathfrak a),  \mathfrak c_{g}(\mathfrak a), 
		\mathfrak c_{h}(\mathfrak a) \right)}\ge \min{\left(\mathfrak c_{f}(\mathfrak a_\infty),  \mathfrak c_{g}(\mathfrak a_\infty), 
		\mathfrak c_{h}(\mathfrak a_\infty) \right)}= 1.
	$$
	
	Now, in view of (\ref{iii111}), we have
	\begin{equation} \label{refree-exam-2}
	1\le d(f, g, h) \cdot \deg{\mathcal C(f, g, h)}\le \dim S_4(\Gamma_0(14)) + g(\Gamma_0(14)) -1 -\epsilon_4 -1=3
	\end{equation}
	since
	$$
	g(\Gamma_0(14))=1.
	$$
	
	Using (\ref{refree-exam-2}), we must have
	$$
	\deg{\mathcal C(f, g, h)}\in \{1, 2, 3\}.
	$$
	
	But $\deg{\mathcal C(f, g, h)}=1$ means that  $\mathcal C(f, g, h)$ is a line which is clearly impossible since $f, g,$ and $h$ are linearly independent. 
	The case $\deg{\mathcal C(f, g, h)}=2$ means that  $\mathcal C(f, g, h)$ is an irreducible conic. Using
	$$
	2d(f, g, h)= d(f, g, h) \cdot \deg{\mathcal C(f, g, h)}\le 3,
	$$
	we must have
	$$
	d(f, g, h)=1
	$$
	This means that $X_0(14)$ is birationally equivalent to the conic  $\mathcal C(f, g, h)$. But  irreducible conic is non-singular. This means that
	$X_0(14)$ isomorphic to a conic. This is a contradiction since conic has genus $0$ while $X_0(14)$ has genus $1$.
	
	Thus,  $\deg{\mathcal C(f, g, h)}=3$. Consequently, $d(f, g, h)=1$ proving the proposition.  
\end{proof}

\vskip .2in

We were also informed by the referee that
$$
g^3 + 3h^3 + f^2 h-f g^2+ gh^2 + fgh = 0.
$$
We remark that Proposition \ref{r-exam-1} implies that
the polynomial
\begin{equation} \label{r-exam-10000}
P= x_1^3 + 3x_2^3 + x^2_0 x_2-x_0 x^2_1+ x_1x^2_2 + x_0x_1x_2.
\end{equation}
is irreducible. The equation has been computed using SAGE.

\vskip .2in
Using other three elements of the basis for $S_4(\Gamma_0(14))$ we obtain the following result:

\begin{Prop} \label{r-exam-2}
	Consider three linearly independent forms from the four dimensional space $S_4(\Gamma_0(14))$  of cusp forms of weight four for $\Gamma_0(14)$: 
	\begin{align*}
	f&= q - 2q^{5} - 4q^{6} - q^{7} + 8q^{8} - 11q^{9} - 12q^{10} + 12q^{11} + 8q^{12} + 38q^{13} \cdots,\\
	g& = q^2- 2q^5-2q^6 + q^7-6q^8 + 12q^{10} + 4q^{11} + 2q^{13}-5q^{14} + 4q^{15}  + \cdots,\\
	h &= q^3-q^5-2q^6-q^7-4q^8 + 6q^9 + 10q^{10}-6q^{11} + 4q^{12}-3q^{13} + \cdots.\\
	\end{align*}
	Then, the map (\ref{map}) is a birational equivalence of $X_0(14)$ and $\calC (f, g, h)$. Moreover, $\deg{\mathcal C(f, g, h)} =4$. 
\end{Prop}
\begin{proof} This case is different that previous one since now $f$ does not have a double zero at $\mathfrak a_\infty$. Consequently, we do not know anything about
	$$\sum_{\mathfrak a\in X_0(14)} 
	\min{\left(\mathfrak c_{f}(\mathfrak a),  \mathfrak c_{g}(\mathfrak a), 
		\mathfrak c_{h}(\mathfrak a) \right)}$$
	besides it is $\ge 0$.  Now, in view of (\ref{iii111}), we have
	$$
	1\le d(f, g, h) \cdot \deg{\mathcal C(f, g, h)}\le \dim S_4(\Gamma_0(14)) + g(\Gamma_0(14)) -1 -\epsilon_4 -1=4.
	$$
	This implies that
	$$
	\deg{\mathcal C(f, g, h)}\le 4.
	$$
	Using SAGE we compute that $P(f, g, h)=0$, where
	\begin{equation} \label{r-exam-10001}
	\begin{aligned}
	P&=
	-3 x_{0}^{2} x_{1}^{2} - 6 x_{0} x_{1}^{3} - 4 x_{1}^{4} + 3 x_{0}^{3} x_{2} + 6 x_{0}^{2} x_{1} x_{2} - 3 x_{0} x_{1}^{2} x_{2} - 2 x_{1}^{3} x_{2} +\\
	&+ 10 x_{0}^{2} x_{2}^{2} +2 x_{0} x_{1} x_{2}^{2} - 21 x_{1}^{2} x_{2}^{2} + 23 x_{0} x_{2}^{3} + 16 x_{1} x_{2}^{3} + 11 x_{2}^{4}.
	\end{aligned}
	\end{equation}
	
	Also, using SAGE system, we checked that this polynomial is irreducible. Hence,
	
	$$
	\deg{\mathcal C(f, g, h)}=4.
	$$
	
	Consequently, we have
	$$
	4 d(f, g, h)=  d(f, g, h) \cdot \deg{\mathcal C(f, g, h)}\le 4.
	$$
	Hence, we have
	$$
	d(f, g, h)=1.
	$$
	This proves the proposition. 
\end{proof}

\vskip .2in

\begin{Prop} \label{r-exam-3}
	Let $N\ge 1$ such that $g(\Gamma_0(N))\ge 3$.   Then, there exists  $f, g,$ and $h$ linearly independent in
	$S_2(\Gamma_0(N))$ such that $d(f, g, h)\cdot  \deg{\mathcal C(f, g, h)}\le g(\Gamma_0(N))+1$.
	In particular, $\deg{\mathcal C(f, g, h)}\le g(\Gamma_0(N))+1$. Moreover, if
	$\left(g(\Gamma_0(N))+1\right)/2< \deg{\mathcal C(f, g, h)}$, then  $X_0(N)$ is birational  to $\calC (f, g, h)$ via the map (\ref{map}).
\end{Prop}
\begin{proof} We use  standard elementary argument (see for example  \cite[Lemma 4.3]{Muic}). We let
        \begin{equation} \label{r-exam-30000}
        W_i=\left\{f\in S_2(\Gamma_0(N)); \  f=0 \ \text{or} \ \nu_{\mathfrak a_\infty}(f)\ge i\right\},
        \end{equation}
        for all $i\ge 1$. Then, all $W_i$ are linear subspaces of  $S_2(\Gamma_0(N))$. Moreover, 
        $$
        S_2(\Gamma_0(N)=W_1\supset W_2\supset W_3\supset \cdots,
        $$
        where  $W_i=W_{i+1}$, or $W_{i+1}$  is of codimension one in $W_i$ for all $i\ge 1$. Now, since
        $\dim S_2(\Gamma_0(N))= g(\Gamma_0(N))$, by counting dimensions, we see that
        $$
        \dim W_{g(\Gamma_0(N))-2} \ge 3
        $$

       Let us select  linearly independent
       $f, g, h\in W_{g(\Gamma_0(N))-2}$. Then, we obtain
       $$
       \sum_{\mathfrak a\in  X_0(N)} 
	\min{\left(\mathfrak c_{f}(\mathfrak a),  \mathfrak c_{g}(\mathfrak a), 
		\mathfrak c_{h}(\mathfrak a) \right)}\ge 	\min{\left(\mathfrak c_{f}(\mathfrak a_\infty),  \mathfrak c_{g}(\mathfrak a_\infty), 
	  \mathfrak c_{h}(\mathfrak a_\infty) \right)}\ge g(\Gamma_0(N))-3.
        $$

        Finally, we have 
	(see (\ref{iii111}))
	\begin{align*}
	d(f, g, h) \cdot \deg{\mathcal C(f, g, h)}&=2g(\Gamma_0(N)) -2 - \sum_{\mathfrak a\in  X_0(N)} 
	\min{\left(\mathfrak c_{f}(\mathfrak a),  \mathfrak c_{g}(\mathfrak a), 
		\mathfrak c_{h}(\mathfrak a) \right)}=\\
	&\le 2g(\Gamma_0(N)) -2 -(g(\Gamma_0(N)) -3)=g(\Gamma_0(N))+1.
	\end{align*}
	This proves the first claim of the proposition. Other claims follow from this immediately. 
\end{proof}

\vskip .2in
We give the following corollaries to Proposition \ref{r-exam-3}:

\vskip .2in
\begin{Cor} \label{r-exam-6}
  Let $N=63$. Then, $g(\Gamma_0(63))= 5$.   A computation of the basis of $S_2(\Gamma_0(63))$ in SAGE implies that the basis
  of $W_3$ (see (\ref{r-exam-30000})) is given by  
	\begin{align*}  
	f\overset{def}{=}& q^3 - q^6 + q^9 - q^{12} - 2q^{15} - q^{18} - q^{21} + 3q^{24} + \cdots,\\
	g\overset{def}{=}&q^4 + q^7 - 4q^{10} + 2q^{13} - 2q^{16} - 4q^{19} + 5q^{22} + \cdots,\\
	h\overset{def}{=}& 2q^5 - q^8 - 3q^{11} - q^{14} + 2q^{17} + q^{23} +\cdots.
	\end{align*}
	 The curve $X_0(63)$ is birational to the curve $\mathcal C(f, g, h)$
	via the map (\ref{map}). The curve the curve $\mathcal C(f, g, h)$ has degree $\deg{\mathcal C(f, g, h)}= g(\Gamma_0(N))+1=6$, and it is defined via irreducible polynomial
	$$
	-2 x_{0}^{4} x_{1}^{2} -  x_{0} x_{1}^{5} + x_{0}^{5} x_{2} + 2 x_{0}^{2} x_{1}^{3} x_{2} + x_{0}^{3} x_{1} x_{2}^{2} -  x_{1}^{4} x_{2}^{2} + 3 x_{0} x_{1}^{2} x_{2}^{3} - 3 x_{0}^{2} x_{2}^{4}
	$$
\end{Cor}
\begin{proof} The equation of the curve is computed in SAGE. Except indicated computations  in SAGE, the claim follows from Proposition \ref{r-exam-3}.
\end{proof}

\vskip .2in
\begin{Cor} \label{r-exam-7}
  Let $N=93$. Then, $g(\Gamma_0(93))= 9$.  A computation of the basis of $S_2(\Gamma_0(93))$ in SAGE implies that the basis
  of $W_7$ (see (\ref{r-exam-30000})) is given by 
	\begin{align*}  
	f\overset{def}{=}& q^7 + q^8 + 2q^9 - 4q^{10} - q^{11} + 2q^{12} + 3q^{13} - \cdots,\\
	g\overset{def}{=}& 2q^8 + 2q^9 - 6q^{10} + 3q^{12} + 5q^{13} - 4q^{14} - 6q^{15} - \cdots,\\
	h\overset{def}{=}& 4q^9 - 4q^{10} - 3q^{11} - q^{12} + q^{13} - 3q^{15} + 2q^{16} + \cdots.
	\end{align*}
	The curve $X_0(93)$ is birational to the curve $\mathcal C(f, g, h)$
	via the map (\ref{map}). The curve $\mathcal C(f, g, h)$ has degree $\deg{\mathcal C(f, g, h)}= g(\Gamma_0(N))+1=10$, and it is defined via irreducible polynomial
	$$
	-30000 x_{0}^{8} x_{1}^{2} + 172400 x_{0}^{7} x_{1}^{3} -  \cdots   + 14065 x_{0} x_{2}^{9} + 355 x_{1} x_{2}^{9} - 1825 x_{2}^{10}.
	$$
\end{Cor}
\begin{proof} The equation of the curve is computed in SAGE. Except indicated computations  in SAGE, the claim follows from Proposition \ref{r-exam-3}.
\end{proof}

\vskip .2in
\begin{Cor} \label{r-exam-8}
  Let $N=110$. Then, $g(\Gamma_0(110))= 15$.  A computation of the basis of $S_2(\Gamma_0(93))$ in SAGE implies that the basis
  of $W_{13}$ (see (\ref{r-exam-30000})) is given by 
	\begin{align*}  
	f\overset{def}{=}& q^{13} + q^{14} - 3q^{16} - 5q^{18} + 5q^{20} - 2q^{21} - q^{22} - \cdots,\\
	g\overset{def}{=}& 2q^{14} - 3q^{16} - q^{17} - 6q^{18} + q^{19} + 6q^{20} - 3q^{21} - \cdots,\\
	h\overset{def}{=}& 3q^{15} + 4q^{16} - 4q^{17} + 7q^{18} + 5q^{19} - 2q^{20} - q^{21}  +\cdots.
	\end{align*}
	The curve $X_0(110)$ is birational to the curve $\mathcal C(f, g, h)$
	via the map (\ref{map}). The curve $\mathcal C(f, g, h)$ has degree $\deg{\mathcal C(f, g, h)}= 15< g(\Gamma_0(N))+1=16$, and it is defined via irreducible polynomial
	\begin{align*} 
	&-198700267941 x_{0}^{13} x_{1}^{2} + 1714521491172 x_{0}^{12} x_{1}^{3} - \cdots  +
	48120 x_{0}^{2} x_{2}^{13}- \\
	& - 91118 x_{0} x_{1} x_{2}^{13} + 43558 x_{1}^{2} x_{2}^{13} + 173 x_{0} x_{2}^{14} - 138 x_{1} x_{2}^{14} + x_{2}^{15}
	\end{align*}
\end{Cor}
\begin{proof} The equation of the curve is computed in SAGE. Except indicated computations  in SAGE, the claim follows from Proposition \ref{r-exam-3}.
\end{proof}

\vskip .2in
Let us explain conjectural  generalization of above corollaries.  
We say that $\Gamma_0(N)$--orbit $\mathfrak a_\infty=\Gamma_0(N).\infty$ is a Weierstrass point for $X_0(N)$ if
$g(\Gamma_0(N))\ge 2$, and there exists a non-zero $f\in S_2(\Gamma_0(N))$ such that $\nu_{\mathfrak a_\infty}(f)\ge g(\Gamma_0(N))+1$.
This a particular case of the much more general definition of a Weierstrass point on a compact Riemann surface \cite[Definition 6.1]{ono}. By the same reference,
if $\mathfrak a_\infty$ is not a Wierestrass point, then there exists a basis $h_1, \ldots, h_g$ of $S_2(\Gamma_0(N))$ such that $\nu_{\mathfrak a_\infty}(h_i)=i$ for $1\le i\le g$.
We have that $W_{g(\Gamma_0(N))-2}$ (see (\ref{r-exam-30000})) has a basis $h_{g-2}$, $h_{g-1}$, $h_g$. 
Obviously, computing the base of $S_2(\Gamma_0(N))$ in SAGE system it is easy to check whether or not $\mathfrak a_\infty$
is a Weierstrass point on $X_0(N)$. Using this method, one check that  $\mathfrak a_\infty$ is not a  Weierstrass point on $X_0(63)$, $X_0(93)$, and $X_0(110)$ (see Corollaries \ref{r-exam-6},
\ref{r-exam-7}, and \ref{r-exam-8}). Also, by \cite{Ogg},  $X_0(63)$, $X_0(93)$, and $X_0(110)$ are not hyperelliptic curves.  We remark that the conditions $g(\Gamma_0(N))= 3$, $\mathfrak a_\infty$ is not a
Wierestrass point for $X_0(N)$, and $X_0(N)$ is not hyperelliptic imply $N\in \{34, 43, 45\}$. Let $f=h_1$, $g=h_2$, and $h=h_3$.  Then, the map  (\ref{map}) is a canonical isomorphism of 
$X_0(N)$ onto $\mathcal C(f, g, h)$. By general theory, the degree of $\mathcal C(f, g, h)$ is $2g(\Gamma_0(N))-2=4$ which is equal to
$g(\Gamma_0(N))+1$. We have computed many more examples of above sort that indicate validity of the following conjecture:

\vskip  .2in
\begin{Conj}\label{conj-s2} Let $N\ge 1$ be such that $g(\Gamma_0(N))\ge 3$, $\mathfrak a_\infty$ is not a Wierestrass point for $X_0(N)$, and $X_0(N)$ is not hyperelliptic.
	Let $f, g, h\in S_2(\Gamma_0(N))$ be such that  $\nu_{\mathfrak a_\infty}(f)=g(\Gamma_0(N))-2$, $\nu_{\mathfrak a_\infty}(g)=g(\Gamma_0(N))-1$, and
	$\nu_{\mathfrak a_\infty}(h)=g(\Gamma_0(N))$. Then, the map (\ref{map}) is birational equivalence, and the curve $\mathcal C(f, g, h)$ has degree
	$\deg{\mathcal C(f, g, h)}\le g(\Gamma_0(N))+1$.
\end{Conj}

\vskip .2in
Let us show that the assumption that $X_0(N)$ is not hyperelliptic in above conjecture is necessary. First,
$\mathfrak a_\infty$ is not a Wierestrass point for $X_0(48)$ since $g(\Gamma_0(48))=3$, and the basis of $S_2(\Gamma_0(48))$ is given by
\begin{align*}
f&=q - 2q^5 + q^9 - 2q^{13} + 2q^{17} - q^{25} + 6q^{29} - 4q^{33} + 6q^{37} - 6q^{41} + \cdots,\\
g&=q^2 - q^6 - 2q^{10} + q^{18} + 4q^{22} - 2q^{26} + 2q^{30} + 2q^{34} - 4q^{38} - 8q^{46}
+ \cdots,\\
h&=q^3 - 4q^{11} - 2q^{15} + 4q^{19} + 8q^{23} + q^{27} - 8q^{31} - 2q^{39} - 4q^{43} + \cdots.
\end{align*}
The corresponding reduced equation of  $\mathcal C(f, g, h)$ is  given by the irreducible polynomial
$-x_1^2 + x_0x_2$. Thus, $\deg{\mathcal C(f, g, h)}= (g(\Gamma_0(48))+1)/2=2$. By \cite{Ogg}, $X_0(48)$ is hyperelliptic.
Then, since the map (\ref{map}) is a canonical map, it has degree two by general theory \cite[Chapter VII, Proposition 2.2]{Miranda}. 
 Thus, it is not a birational equivalence.

\vskip .2in
Let us show that the assumption that  $\mathfrak a_\infty$ is not a Wierestrass point for $X_0(N)$  in above conjecture is necessary. 
The curve $X_0(72)$ is not hyperelliptic \cite{Ogg}, $g(\Gamma_0(72))=5$, and $\mathfrak a_\infty$ is a
Wierestrass point for $X_0(72)$ since the basis  of $W_{3}$ is given by
\begin{align*}
f&= q^3 - q^9 - 2q^{15} + q^{27} + 4q^{33} - 2q^{39} + \cdots,\\
g&=q^5 - 2q^{11} - q^{17} + 4q^{23} - 3q^{29} + \cdots,\\
h&=q^7 - q^{13} - 3q^{19} + q^{25} + 3q^{31} + 4q^{37} +  \cdots.
\end{align*}
The reduced equation of $\mathcal C(f, g, h)$ is  given by the irreducible polynomial $-x_{0} x_{1}^{2} + x_{0}^{2} x_{2} - 2 x_{1} x_{2}^{2}$. Hence,
$\deg{\mathcal C(f, g, h)}=3$.
Now,  since the proof of Proposition \ref{r-exam-3} implies that
$$
3\cdot d(f, g, h) = d(f, g, h)\cdot  \deg{\mathcal C(f, g, h)}\le g(\Gamma_0(74))+1=6,
$$
we obtain $d(f, g, h)\le 2$. We have the following proposition.

\begin{Prop} \label{r-exam-300000}
  Under above assumptions, we have $d(f, g, h)= 2$. 
  \end{Prop}
\begin{proof}
First, using the reduced equation, it easy to check that  $(1:0:0)$ is a non--singular point on $\mathcal C(f, g, h)$. Then, the implicit function theorem implies
that the local coordinate in that point is $x_1/x_0$. Next, (\ref{map}) maps $\mathfrak a_\infty$ onto $(1:0:0)$.  The local coordinate at  $\mathfrak a_\infty$ is $q$. Thus,
in terms of local coordinates,  the map (\ref{map}) is given by
\begin{equation} \label{r-exam-300001}
q\longmapsto \frac{g}{f}=q^2 \frac{1 - 2q^{6} - q^{12} + 4q^{18} - 3q^{24} + \cdots}{1 - q^6 - 2q^{12} + q^{24} + 4q^{30} - 2q^{36} + \cdots}.
\end{equation}
Now, we use deeper properties of  the proof of \cite[Theorem 1-4]{Muic2}. Since $\mathfrak a_\infty$ is mapped onto a non--singular point, 
the paragraph before the statement of \cite[Lemma 3-4]{Muic2}, shows that 
the multiplicity of the map (\ref{map}) at point $\mathfrak a_\infty$  is well--defined. Using (\ref{r-exam-300001}), we see that the multiplicity at  $\mathfrak a_\infty$
is at least two. But since $d(f, g, h)\le 2$,  \cite[Lemma 3-4]{Muic2} implies that  the multiplicity at  $\mathfrak a_\infty$ is exactly two, and  $d(f, g, h)=2$.
\end{proof}

\vskip .2in
On the other hand, $X_0(54)$ is not hyperelliptic by \cite{Ogg}, $g(\Gamma_0(54))=4$, and $\mathfrak a_\infty$ is a Wierestrass point for $X_0(54)$ since the basis  of $W_{2}$ is given by
\begin{align*}
f&= q^2 - 2q^8 - q^{14} + 5q^{26} + 4q^{32} - 7q^{38} +  \cdots,\\
g&=q^4 - q^{10} - 3q^{13} - q^{16} + 3q^{19} + q^{22} + 3q^{25} - q^{28} + 3q^{31} - 3q^{37} + \cdots,\\
h&=q^5 - q^8 - q^{11} + q^{20} - 2q^{23} + 3q^{26} + 2q^{29} + q^{32} - q^{35} - 3q^{38} +  \cdots.
\end{align*}
The corresponding reduced equation of  $\mathcal C(f, g, h)$ is  given by the irreducible polynomial
$$
-x_{0}^{2} x_{1}^{3} + 3 x_{0} x_{1}^{3} x_{2} + x_{0}^{3} x_{2}^{2} - 3 x_{1}^{3} x_{2}^{2} - x_{0}^{2} x_{2}^{3} + 3 x_{0} x_{2}^{4} - 3 x_{2}^{5}.
$$
Thus, $\deg{\mathcal C(f, g, h)}= g(\Gamma_0(54))+1$.  Now, Proposition \ref{r-exam-3} impies that the map (\ref{map}) is a birational equivalence.

\section{Estimates for Primitive Elements}
\label{epe}

In this Section we will look for applications and improvements on  Theorem \ref{thm-3}
in the case when the subspace $W$ has dimension $4$ (see  Theorem \ref{thm-3} for the notation and assumptions on $W$).
We use the Primitive Element Theorem of finite separable field extensions in the form stated in \cite[Section 6.10]{w}.
We start by recalling certain facts from  \cite[Section 6.10]{w}.

\vskip .2in

Let $K\subset L$ be a finite algebraic field extension.
We assume that $L$ is  generated over $K$ by two  elements $\alpha$ and $\beta$.
We are interested in  the field of characteristic zero, but we work in
a greater generality.  We assume that $K$ is infinite and $K\subset L$ 
is separable. 	By the general  theory  \cite[Section 6.10]{w}, 
since $K$ is infinite,  there exists a primitive element of the field extension $K\subset L$ of the form 
$\alpha+c\beta$ for some 
$c\in K$. We just need to take $c\in K$ different than
all 
$$
\frac{\alpha_i-\alpha}{\beta- \beta_j}, \ \ 1\le i \le m,\ \ 2\le j\le m
$$
where $\alpha_1=\alpha, \ldots,  \alpha_m$, and $\beta_1=\beta, \ldots,  \beta_n$  
are all conjugates of $\alpha$ and $\beta$
in some  algebraic closure of $K$ containing $L$.  Let us recall a simple argument  \cite[Section 6.10]{w}.
Let $P$ and $Q$ be  irreducible polynomials of $\alpha$ and $\beta$ over $K$, respectively. 
We write them in the form:

\begin{equation}\label{alg-0}
\begin{aligned}
P(X) &=a_mX^m+ a_{m-1}X^{m-1}+\cdots +a_1X+a_0= a_m(X-\alpha_1)\cdots (X-\alpha_m)\\
Q(X) &=b_nX^n+ b_{n-1}X^{n-1}+\cdots +b_1X+b_0=b_n(X-\beta_1)\cdots (X-\beta_m).
\end{aligned}
\end{equation}

Select $c$ as above and let $\gamma =\alpha +c\beta$. Then $\beta$ is a common  root of $P(\gamma-cT)$ and 
$Q(T)$ which are the polynomials with coefficients in $K(\gamma)$. Since $\beta$ is separable, all roots of $Q$ are 
simple, and because of our assumption $\beta$ is the  only common root. So, computing greatest common divisor, we conclude 
that $X-\beta$ has coefficients in $K(\gamma)$. Hence $\beta \in K(\gamma)$.
So, $\alpha=\gamma -c\beta\in K(\gamma)$. The claim follows. 
In this classical argument $\alpha$ is not necessarily separable but  we would like to explain how to compute such $c$ 
without assuming that we know all roots. For this we need the assumption that $\alpha$ is separable.

\vskip .2in 
The following lemma is an improvement of above argument in the case of finite extensions of algebraic function
fields.  The case of number fields is considered in \cite{ynt}.

\vskip .2in 
\begin{Lem}\label{epe-1} Let $K=k(T)$, a field of rational functions in one variable $T$ over a field $k$.
	Consider a finite separable algebraic field extension $K\subset L$.
	We assume that $L$ is  generated over $K$ by two  elements $\alpha$ and $\beta$.
	Let $P$ and $Q$ be irreducible polynomials of $\alpha$ and $\beta$ over $K$, respectively. Clearing denominators, 
	we can write them in the form:	
	\begin{align*}
	P(X, T) &=a_m(T)X^m+ a_{m-1}(T)X^{m-1}+\cdots +a_1(T)X+a_0(T),\\
	Q(X, T) &=b_n(T)X^n+ b_{n-1}(T)X^{n-1}+\cdots +b_1(T)X+b_0(T) \in k[X, T], 
	\end{align*}
	where  $a_i, b_j\in k[T]$.  Assume that  $\lambda \in k$ is selected such that the following holds:
        \begin{itemize}
          \item[(a)] $P(X, \lambda)$ and $Q(X, \lambda)$ have
	    degrees $m$ and $n$ as polynomials in $k[X]$, respectively, and
            \item[(b)] $Q(X, \lambda)$ considered as a polynomial in $X$ with coefficients in $k$, has simple roots in some (hence, any) algebraic closure of $k$.
        \end{itemize}
        Then, if we  write $\overline{\alpha}_1, \ldots, \overline{\alpha}_m$ (resp., $\overline{\beta}_1, \ldots, \overline{\beta}_n$) for all 
	roots of $P(X, \lambda)$ (resp., $Q(X, \lambda)$) in some algebraic closure of $k$, then for 
	$c\in k$ different than all
	$$
	\frac{\overline{\alpha}_i-\overline{\alpha}_{i_1}}{\overline{\beta}_{j_1}- \overline{\beta}_j}, 
	\ \ 1\le i, i_1 \le m,\ \ 1\le j, j_1\le m, \ \
	j\neq j_1,
	$$
	we have that 
	$\alpha+c\beta$ is a primitive element for the extension $K\subset L$.
\end{Lem}
\begin{proof} Consider  $k$--algebra $k[T]_{\lambda}$  of all $a/b$, $a, 
	b\in k[T]$, $b(\lambda)\neq 0$. Let $\mathfrak m_\lambda$  be the maximal ideal in
	$k[T]_{\lambda}$ consisting of all functions vanishing at $\lambda$.
	Let $\overline{k(T)}$ be the  algebraic closure of $K=k(T)$ containing $L$. 
	Let $\overline{k[T]}_{\lambda}$ be the integral closure of $k[T]_{\lambda}$ in $\overline{k(T)}$.
	Let $\mathfrak M$ be a maximal ideal in $\overline{k[T]}_{\lambda}$ lying above $\mathfrak m_\lambda$. Then, 
	$$
	\overline{k[T]}_{\lambda}/\mathfrak M
	$$
	is the algebraic closure of
	
	$$
	k[T]_{\lambda}/\mathfrak m_\lambda.
	$$
	
	Let us write in $\overline{k(T)}[X]$
	\begin{equation}\label{epe-2}
	\begin{aligned}
	P(X, T) &=a_m(T)(X-\alpha_1)\cdots (X-\alpha_m)\\
	Q(X, T) &=b_n(T)(X-\beta_1)\cdots (X-\beta_n).
	\end{aligned}
	\end{equation}
	Then, clearly  $\alpha=\alpha_1, \ldots, \alpha_m, \beta=
	\beta_1, \ldots, \beta_n$ are integral over $k[T]_{\lambda}$. Hence,  
	they belong to $\overline{k[T]}_{\lambda}$. Let $\Lambda$ be the reduction homomorphism 
	$$
	\overline{k[T]}_{\lambda}\longrightarrow \overline{k[T]}_{\lambda}/\mathfrak M.
	$$
	Applying $\Lambda$ to  (\ref{epe-2}), we may assume that 
	
	\begin{equation}\label{epe-3}
	\begin{aligned}
	&\Lambda(\alpha_i)= \overline{\alpha}_i, \ \ 1\le i \le m, \\
	&\Lambda(\beta_i)= \overline{\beta}_i, \ \ 1\le i \le n.
	\end{aligned}
	\end{equation}
	
	Now, 
	$$
	c({\overline{\beta}_{j_1}- \overline{\beta}_j)\neq 
		\overline{\alpha}_i-\overline{\alpha}_{i_1}}
	$$
	implies that 
	$$
	c(\beta_{j_1}- \beta_j)\neq \alpha_i-\alpha_{i_1},
	$$ 
	for all $1\le i, i_1 \le m$, $1\le j, j_1\le m$, $j\neq j_1$. By the results recalled in the beginning of this Section, we obtain 
	that $\alpha+c\beta$ is primitive for the extension $K\subset L$. \end{proof}

\vskip .2in

To apply Lemma \ref{epe-1}, we need the following lemma:

\begin{Lem}\label{epe-4}
  Let $f(X)=a_nX^n+a_{n-1}X^{n-1}+\dots+a_0=a_n\prod_{i=1}^n(X-\alpha_i) \in  \mathbb Z[X]$. be a polynomial of degree $n\ge 1$. Then, we have the following:
  $$
  |\alpha_i-\alpha_j|<2 L(f).
  $$
  If, in addition,  $f$ has no multiple roots  i.e., $\alpha_i\neq \alpha_j$ for $i\neq j$, then   we have the following: 
	$$
	|\alpha_i-\alpha_j|>\sqrt{3}n^\frac{-(n+2)}{2}L(f)^{-(n-1)}.
	$$
        Here  $L(f)=|a_n|+|a_{n-1}|+\cdots+|a_1|+|a_0|$.
\end{Lem}
\begin{proof} The first bound is elementary and well-known.  It follows from the  Rouch\' e theorem
  in Complex analysis. We sketch the argument. If $R>0$ is selected such that 
	$$
	|a_n| R^n> \sum_{i=0}^{n-1}  |a_i| R^i,
	$$ 
	then all roots of $f$ belong to $|z|<R$. We may select
	$$
	R=\max{\left\lbrace  1, \frac{|a_0|+\cdots+ |a_{n-1}|}{|a_n|},\right\rbrace}\leq \max{\left\lbrace 1,\frac{L(f)}{|a_n|}\right\rbrace }.
        $$
        Now, we apply that $f$ has integral coefficients: $|a_n|\ge 1$. 
	
	The second bound is more complicated. It  can be found in \cite{Mahler}.
\end{proof}

Finally, the main result of the present section is the following proposition.

\begin{Prop}\label{epe-5} Assume that  $m\ge 2$ is an even integer. Let $W\subset M_m(\Gamma)$, $\dim W=4$,  be a subspace which determines
  the field of rational functions $\mathbb C(\mathfrak R_\Gamma)$ (see Definition \ref{gfr-def}). Select a basis  $\{f=f_0, g=f_1, f_2,f_3\}$ of $W$.
  We assume that all $f_i$ has integral $q$--expansions.    Then, there exists an {\bf explicitly computable}
   $c_0\in \mathbb Z$ such that for all $c\in \mathbb Z$, $|c|\ge c_0$, $\mathfrak R_\Gamma$ is birationally equivalent to  $\calC (f, g, h_c)$  via the
        map (\ref{map}) with $h=h_c$, where
  $h_c\overset{def}{=} f_2+ cf_3$. 
  \end{Prop}
\begin{proof} We use the notation of Lemma \ref{epe-1}. Put
  $$
  K\overset{def}{=}\mathbb C(g/f) \ \ \text{and} \ \  L\overset{def}{=}\mathbb C (\mathfrak R_\Gamma)=\mathbb C(f_1/f_0,  f_2/f_0, f_{3}/f_0)=\mathbb C(g/f,  f_2/f, f_{3}/f).
 $$
  In the notation of Lemma \ref{epe-1}, $\alpha= f_2/f$ and $\beta= f_3/f$.

  Next,  by  the argument used in the  proof of Corollary \ref{exist-1}, since we assume that  $f_i$ has integral $q$--expansions, we may assume that  curves $\mathcal C (f, g, f_2)$
  and $\mathcal C (f, g, f_3)$ have have their reduced equations with coefficients in $\mathbb Z$. Dehomogenizing the reduced equations, we obtain two polynomials
  $P(X, T)$ and $Q(X, T)$ in $\mathbb Z[X, T]$ such that
  $$
  P\left(f_2/f, g/f\right)=0 \ \ \text{and} \ \  Q\left(f_3/f, g/f\right)=0.
  $$
 They are both irreducible as polynomials  in $\mathbb Q[X, T]$.

  We select $\lambda \in \mathbb Z$ as required by (a) and (b) in Lemma \ref{epe-1}.
For (b), one might compute the discriminant of $Q(X, T)$ with respect to $X$ i.e., the resultant of $Q(X, T)$ and its derivative with respect to $X$.
Since $Q(X, T)$ is irreducible in  $\mathbb Q[X, T]$, the  resultant $R(T)$ is a polynomial in $\mathbb Z[T]$ not identically equal to zero.  Now, in the notation used in
Lemma \ref{epe-1},
we select $\lambda\in \mathbb Z$ such that $a_m(\lambda)b_n(\lambda)R(\lambda)\neq 0$. 

Finally, one can apply Lemma \ref{epe-1} combined with bounds of Lemma \ref{epe-4} applied to polynomials $P(X, \lambda)$ and  $Q(X, \lambda)$ in $\mathbb Z[X]$.
The details are left to the reader as an easy
exercise.
\end{proof}

\vskip .2in
The bound mentioned in the proof of Proposition \ref{epe-5} is not very optimal as we observed by various computations using SAGE. The problem is with the Mahler's
estimate (see  the second inequality in
Lemma \ref{epe-3}). But in some cases we can obtain good results. We include the following example:

\begin{Prop}\label{epe-6}
  Consider the four dimensional space $W\overset{def}{=}S_4(\Gamma_0(14))$  of cusp forms of weight four for $\Gamma_0(14)$. It has a basis:
  	\begin{align*}
	f= f_0 &= q - 2q^{5} - 4q^{6} - q^{7} + 8q^{8} - 11q^{9} - 12q^{10} + 12q^{11} + 8q^{12} + 38q^{13} \cdots,\\
	g = f_1 &= q^2- 2q^5-2q^6 + q^7-6q^8 + 12q^{10} + 4q^{11} + 2q^{13}-5q^{14} + 4q^{15}  + \cdots,\\
	f_2 &= q^3-q^5-2q^6-q^7-4q^8 + 6q^9 + 10q^{10}-6q^{11} + 4q^{12}-3q^{13} + \cdots,\\
        f_3 &= q^4-2q^5 + q^7 + q^8-4q^{10} + 4q^{11}-2q^{12} + 2q^{13} + 2q^{14} + 4q^{15}-5q^{16} + \cdots.
	\end{align*}
        Put $h_c\overset{def}{=}f_2+cf_3$, $c\in \mathbb Z$, as in the statement of Proposition  \ref{epe-5}. Then, $X_0(14)$ is birationally equivalent  to  $\calC (f, g, h_c)$ via the
        map (\ref{map}) with $h=h_c$ for $|c|\ge 7$. 
        \end{Prop}
\begin{proof} We apply Proposition \ref{epe-5}. First, $W=S_4(\Gamma_0(14))$  determines
  the field of rational functions on $X_0(4)$ by Proposition  \ref{r-exam-1}. Next, we  recall form the proof of Proposition \ref{epe-5}
  that  $\alpha= f_2/f$ and $\beta= f_3/f$. Then, (\ref{r-exam-10001}) imply that we have
  $$
  P(X, T)= 11X^4 + (23+16T)X^3 + (10+2T-21T^2)X^2 + (3+6T-3T^2 -2T^3)X + (-3T^2-6T^3-4T^4),
  $$
  and using similar computation in SAGE we obtain:
  $$
    Q(X, T)=11X^4 -(18 + 24T)X^3-(5 +9T+3T^2)X^2 -(1 +2T+ 4T^2+6T^3)X+ T^4.
  $$
  One sees that we can select $\lambda=0$ to ensure that the assumptions (a) and (b) of Lemma \ref{epe-1} hold. 

  Using SAGE, it is easy to compute roots of both polynomials $P(X, 0)$ and $Q(X, 0)$ with required precision. From that we obtain for the roots $\alpha_i$ of $P(X, 0)$ the bound 
  $$
  \left|\alpha_i- \alpha_j\right|< 1.638,
  $$
  and for the roots $\beta_i$ of $Q(X, 0)$ we have
  $$
  \left|\beta_i- \beta_j\right|> 0.2595.
  $$
  Thus, we have 
  	$$
	\left|\frac{\alpha_i- \alpha_{i_1}}{\beta_{j_1}- \beta_j}\right|< 6.308 < 7,
        $$
        for all $1\le i, i_1 \le 4$, and  $1\le j, j_1\le 3$, $j\neq j_1$.
        Thus, by Lemma \ref{epe-1}, we can select $|c|\ge 7$, $c\in\mathbb Z$, to obtain the claim of the proposition. 
  \end{proof}

\vskip .2in
The estimate given by Proposition \ref{epe-6} is quite good. They depend on our choice of $\lambda=0$. It is also possible to select for example $\lambda=1$ which results in a weaker estimate
$|c|\ge 9$, or $\lambda=-1$ which results in a better estimate $|c|\ge 4$.

\vskip .2in 
We remark that the methods used in the proof of Proposition \ref{r-exam-2} imply that
	$$
	d(f, g, h_c) \cdot \deg{\mathcal C(f, g, h_c)}\le 4, \ \ c\in \mathbb Z. 
	$$
  As in Section \ref{r-exam}, we use SAGE to compute the reduced equation of $\deg{\mathcal C(f, g, h_c)}$, for $c=0, \ldots, 6$. The degree is always
        equal to $4$. In particular, in view of above estimate, we see that $X_0(14)$ is birational to  $\calC (f, g, h_c)$ via the
        map  (\ref{map}) with  $h=h_c$ for $c\in \{0, \ldots, 6\}$. Combining with Proposition \ref{epe-6}, we obtain the following corollary:

        \begin{Cor}\label{epe-6-cor}
           $X_0(14)$ is birationally equivalent to  $\calC (f, g, h_c)$ via the
          map  (\ref{map}) with  $h=h_c$, for all integers $c\ge 0$.
        \end{Cor}
        
 It seems that this result can not be established by the methods of 
 Section \ref{r-exam}.  This shows the usefulness of the methods of the present section.

\section{The Trial Method for Primitive Elements}\label{trial}

Let $W\subset S_m(\Gamma)$, $m\geq 2$, be a non-zero subspace that determines the field of rational functions $\bbC(\mathfrak{R}_\Gamma)$ (see Definition \ref{gfr-def}).
Furthermore, we assume that  $\dim W=s\geq 4$.  Let $f_0, \ldots, f_{s-1}$ be a basis of $W$. We let $f=f_0$ and $g=f_1$. Then, 
Theorem \ref{thm-3} guarantees that in $W$ we can find infinitely many forms $h$ such that  $\mathfrak{R}_\Gamma$ is birationally equivalent to $\mathcal{C}(f,g, h)$ via (\ref{map}). But the proof does not provide
a computable manner of determining at least one such $h$. In this section, we present a simple algorithm for this. We adapt  to our case the trial method,
commonly used in the cases of algebraic number fields, \cite{ynt},
where an element that is chosen from a certain subset of the field extension is tested for being primitive. 

\vskip .2in 
As in the proof of Lemma \ref{p2-8}, we denote $K\overset{def}{=}\Bbb C(f/g)$, and 
$$
L\overset{def}{=}\Bbb C (\mathfrak R_\Gamma)=\Bbb C(f_1/f_0, \ f_2/f_0,\ \dots, \ f_{s-1}/f_0)=\Bbb C(f/g, \ f_2/f, \ \dots, \ f_{s-1}/f).
$$
We observe $L$ is a finite algebraic extension of $K$, and we have the following:
$$
L=K(f_2/f_0,\ \dots, \ f_{s-1}/f_0).
$$

\vskip .2in 

We are interested in finding a primitive element of $L$ over $K$ which has the form of linear combination of the generators $f_2/f_0,\dots,f_{s-1}/f_0$. From the proof of Lemma \ref{p2-8},
we know that the coefficients of this linear combination must be from a Zariski open set in $\Bbb C^{s-2}$, and since $\Bbb Z^{s-2}$ is Zariski dense in $\Bbb C^{s-2}$, we can find a $\mathbb Z$--linear
combination which is primitive for $L$. The trial method consist of testing various  $\mathbb Z$--linear combinations for the condition of being primitive element.  

\vskip .2in
For $a\overset{def}{=}(a_2,a_3,\dots,a_{s-1})\in \mathbb Z^{s-2}$, we let
\begin{equation}\label{a-1000}
h\overset{def}{=}h_a\overset{def}{=}a_2f_2/f_0+\cdots+a_{s-1}f_{s-1}/f_0\in L.
\end{equation}

\vskip .2in
Since, by our assumption $W\subset S_m(\Gamma)$, we have
$$
d(f, g, h) \cdot \deg{\mathcal C(f, g, h)} \le \dim S_m(\Gamma) + g(\Gamma) -1  -\epsilon_m,
$$
using (\ref{iii111}). Thus, if
\begin{equation}\label{test}
\deg{\mathcal C(f, g, h)} > \frac{\dim S_m(\Gamma) + g(\Gamma) -1  -\epsilon_m}{2},
\end{equation}
then we obtain
$$
d(f, g, h)=1.
$$
This means that $\mathfrak{R}_\Gamma$ is birationally equivalent to $\mathcal{C}(f,g, h)$ via (\ref{map}).

\vskip .2in
We organize $(s-2)$--tuples in $\mathbb Z^{s-2}$ as follows:
$$
S_M\overset{def}{=}\left\{a_2f_2/f_0+\cdots+a_{s-1}f_{s-1}/f_0; \ \ a_i\in \mathbb Z,\ \ 2\leq i\leq s-1, \ \sum_{i=2}^{s-1}|a_i|= M \right\}, 
$$
for all $M\in \mathbb Z_{\ge 1}$. For $M\ge 1$, we order elements of $S_M$ using the lexicographical order.

\vskip .2in

In order to apply this simple method, we perform the following algorithm which stops after finitely many steps:

\begin{itemize}
\item[(1)] Let $M=1$. Repeat the following:

\item[(2)] For $a\in S_M$, we repeat the following: compute $\deg{\mathcal C(f, g, h)}$, and  test (\ref{test}) for $h=h_a$. If (\ref{test}) holds, then the algorithm stops. OUTPUT: $h$ such that
  $h/f$ is a primitive element for the extension $K\subset L$.  

  \item[(3)] Increase $M$ by one, and return to step (2). 
\end{itemize}

\vskip .2in
Let $\Gamma=\Gamma_0(N)$ such that $g(\Gamma_0(N))\ge 4$, and $X_0(N)$ is not hyperelliptic \cite{Ogg} (or Introduction). Then, it is well--known that $S_2(\Gamma_0(N))$
determines the field of rational functions on $X_0(N)$. Since also its dimension is equal to  $g(\Gamma_0(N))\ge 4$, we may select $W=S_2(\Gamma_0(N))$.  In this case, the inequality
(\ref{test}) is

\begin{equation}\label{test-1}
\deg{\mathcal C(f, g, h)} > g(\Gamma_0(N)) -1. 
\end{equation}

As an example, we consider the case $N=72$. Then,  $g(\Gamma_0(72))=5$, and we may take
\begin{align*}
f=f_0&= q^3 - q^9 - 2q^{15} + q^{27} + 4q^{33} - 2q^{39} + \cdots, \\
g=f_1&=q^5 - 2q^{11} - q^{17} + 4q^{23} - 3q^{29} + \cdots, \\
f_2&=  q^7 - q^{13} - 3q^{19} + q^{25} + 3q^{31} + 4q^{37} + \cdots,\\
f_3&=q - 2q^{13} - 4q^{19} - q^{25} + 8q^{31} + 6q^{37} + \cdots,\\
f_4&= q^2 - 4q^{14} + 2q^{26} + 8q^{38} + \cdots, \\
\end{align*}

The example $N=72$ already appears in Section \ref{r-exam} in the text after Conjecture \ref{conj-s2} (see Proposition \ref{r-exam-300000}). Applying above algorithm, we obtain the following:
\begin{itemize} 
\item[(1)] For $M=1$, we have three cases in their lexicographical order $a=(0, 0, 1)$,  $(0, 1, 0)$, and $(1, 0, 0)$. We have $\deg{\mathcal C(f, g, h_a)}=3$, $2$, and $3$,
  respectively. In any case,  $\deg{\mathcal C(f, g, h_a)}\le g(\Gamma_0(72)) -1=4$. So, we go to the next step.
\item[(2)] For $M=2$, in the lexicographical order, we have the following:
  \begin{itemize}
\item[1.] $a=(0, 0, 2)$, $\deg{\mathcal C(f, g, h_a)}=3\le g(\Gamma_0(72)) -1=4$;
\item[2.] $a=(0, 1, 1)$, $\deg{\mathcal C(f, g, h_a)}=3\le 4$;
\item[3.] $a=(0, 2, 0)$, $\deg{\mathcal C(f, g, h_a)}=2\le 4$;
\item[4.] $a=(1, 0, 1)$, $\deg{\mathcal C(f, g, h_a)}=7> 4$; STOP.
  \end{itemize}
\end{itemize}

\vskip .2in
  Hence, the map (\ref{map}) with $h=h_{(1, 0, 1)}$ is  a birational equivalence of $X_0(72)$ and $\mathcal C(f, g, h_{(1, 0, 1)})$. The reduced equation of $\mathcal C(f, g, h_{(1, 0, 1)})$ is given by the
  irreducible polynomial
  \begin{align*}
    &  x_{0}^{7} - 4 x_{0}^{6} x_{1} - 3 x_{0}^{4} x_{1}^{3} - 8 x_{0}^{3} x_{1}^{4} - x_{0}^{2} x_{1}^{5} - 4 x_{0}x_{1}^{6} - 4 x_{1}^{7} - 4 x_{0}^{5} x_{1} x_{2} + \\
    &+2x_{0}^{3} x_{1}^{3} x_{2} - 4 x_{0}^{2} x_{1}^{4} x_{2} - x_{0}^{4} x_{1} x_{2}^{2} + 8 x_{0}^{3} x_{1}^{2} x_{2}^{2} - 4 x_{0} x_{1}^{4} x_{2}^{2} + 8 x_{1}^{5}x_{2}^{2} +\\
    &+4 x_{0}^{2} x_{1}^{2} x_{2}^{3} - 4x_{1}^{3} x_{2}^{4}.
\end{align*}

\end{document}